\documentclass[12pt]{amsart}

\usepackage{amsfonts,amsthm,amsmath,amssymb,amscd,mathrsfs,tikz}
\usepackage{graphics}
\usepackage{indentfirst}
\usepackage{cite}
\usepackage{latexsym}
\usepackage[dvips]{epsfig}

\usepackage{color}
\usepackage{amssymb,amsmath,bm}

\newtheorem{theorem}{Theorem}[section]
\newtheorem{remark}{Remark}[section]
\newtheorem{definition}{Definition}[section]
\newtheorem{lemma}[theorem]{Lemma}
\newtheorem{pro}[theorem]{Proposition}
\newtheorem{cor}[theorem]{Corollary}

\renewcommand{\div}{{\rm div }}

\newcommand{\bt}{\begin{theorem}}
\newcommand{\bl}{\begin{lemma}}
\newcommand{\el}{\end{lemma}}
\newcommand{\et}{\end{theorem}}

\newcommand{\bn}{\begin{eqnarray}}
\newcommand{\en}{\end{eqnarray}}
\newcommand{\bnn}{\begin{eqnarray*}}
\newcommand{\enn}{\end{eqnarray*}}

\newcommand{\ba}{\begin{aligned}}
\newcommand{\ea}{\end{aligned}}
\newcommand{\be}{\begin{equation}}
\newcommand{\ee}{\end{equation}}

\newcommand{\p}{\partial}

\newcommand{\R}{\mathbb{R}}

\newcommand{\Bf}{{\boldsymbol{f}}}
\newcommand{\Bv}{{\boldsymbol{v}}}
\newcommand{\oBv}{\overline{{\boldsymbol{v}}}}

\newcommand{\Bn}{{\boldsymbol{n}}}

\newcommand{\Bu}{{\boldsymbol{u}}}
\newcommand{\oBu}{\overline{{\boldsymbol{u}}}}
\newcommand{\Be}{{\boldsymbol{e}}}

\newcommand{\Bh}{{\boldsymbol{h}}}

\newcommand{\BU}{\boldsymbol{U}}

\newcommand{\Bg}{{\boldsymbol{g}}}
\newcommand{\Bw}{{\boldsymbol{w}}}

\newcommand{\Ba}{{\boldsymbol{a}}}

\newcommand{\Bp}{{\boldsymbol{\phi}}}

\begin{document}

\title[Two-dimensional flows with Dirichlet boundary condition]
{On the asymptotic behavior of solutions to the steady Navier-Stokes system in two-dimensional channels}

\author{Han Li}
\address{School of Mathematical Sciences, Soochow University, No. 1 Shizi Street, Suzhou, China}
\email{hlihli@stu.suda.edu.cn}

\author{Kaijian Sha}
\address{School of mathematical Sciences, Shanghai Jiao Tong University, 800 Dongchuan Road, Shanghai, China}
\email{kjsha11@sjtu.edu.cn}

\begin{abstract}
In this paper, we investigate the incompressible steady Navier-Stokes system with no-slip boundary condition in a two-dimensional channel. Given any flux, the existence of solutions is proved as long as the width of cross-section of the channel grows more slowly than the linear growth. Furthermore, if the flux is suitably small, the solution is unique even when the width of the channel is unbounded. Finally, based on the estimate of  Dirichlet norm  on the truncated domain, one could obtain the pointwise decay rate of the solution for arbitrary flux.
\end{abstract}


\thanks{Updated on \today}

\maketitle

\section{Introduction}
The famous Leray problem in a channel $\Omega$ with straight outlets, pioneered by Leray in 1950s,  is to study the well-posedness of the steady Navier-Stokes system
\begin{equation}\label{NS}
\left\{
\begin{aligned}
&-\Delta \Bu+\Bu\cdot \nabla \Bu +\nabla p=0 ~~~~&\text{ in }\Omega,\\
&{\rm div}~\Bu=0&\text{ in }\Omega,
\end{aligned}\right.
\end{equation}
supplemented with no-slip boundary condition
\begin{equation}\label{BC}
   \Bu=0  \text{ on }\partial\Omega,
\end{equation} 
and the far field constraint
\begin{equation}\label{far field}
	\Bu \to \BU \text{ as }|x_1|\to \infty.
\end{equation}
Here the unknown function $\Bu=(u_1,\cdots,u_N)$ $(N=2,3)$ is the velocity and $p$ is the pressure, $\BU$ is the shear flow associated to the straight outlets. For example, if $\Omega$ is a two-dimensional channel satisfying
\[\Omega\cap \{(x_1,x_2):~x_1>0\}=\{(x_1,x_2):~ x_1>0,~x_2\in(-1,1)\},\] 
 then $\BU=\frac{3}{4}\Phi(1-x_2^2)\Be_1$ is the Poiseuille flow, where the constant $\Phi$ is called the flux of the flow. Without loss of generality, we always assume that $\Phi$ is nonnegative in this paper.

The major breakthrough for the Leray problem in infinitely long channels was made by Amick \cite{A1,A2}, Ladyzhenskaya and Solonnikov \cite{LS}. It was proved in \cite{A1, LS} that Leray problem is solvable as long as the flux is small. Actually, the existence of solutions of \eqref{NS} in a nozzle
with arbitrary flux  was also proved in \cite{LS}. However, the far field behavior and the uniqueness of such solutions are not clear when the flux is large. To the best of our knowledge, there is no result on the far field behavior of solutions of steady Navier-Stokes system with large flux except for the axisymmetric solutions in a pipe studied in \cite{WX1,WX2}. One may refer to \cite{AF,Ga, Horgan,HW} and the references therein for more results on the asymptotic behavior of solutions to Leray problem.

In fact, Leray problem can be generalized to the  pipe-like domains, whose outlets may not necessary be straight, provided that  the far field constraint \eqref{far field} is replaced by the following flux constraint
\begin{equation}\label{flux constraint}
	\int_{\Sigma(t)}u_1 \,ds=\Phi \text{ for any }t\in \R,
\end{equation}
where $\Sigma(t)=\{x\in \Omega:~x_1=t\}$. Note that the far field constraint \eqref{far field} implies automatically the flux constraint \eqref{flux constraint} when the outlets of the pipe are straight. 

Due to its importance in both mathematical and practical application, the well-posedness of the Navier-Stokes system in pipe-like domains has received special attention in the past 40 years. In \cite{AF}, Amick and Fraenkel studied the well-posedness of the problem \eqref{NS}-\eqref{BC} and \eqref{flux constraint} in two-dimensional channels of various types via the technique of conformal transformation. For channels that  widen strongly at infinity, it is proved that given any flux, the problem \eqref{NS}-\eqref{BC} and \eqref{flux constraint} has a classical solution whose velocity tends to zero at far field. However, their attempts on a rate of decay for the velocity in this case have been wholly unsuccessful. On the other hand, if the channel widen feebly at far field, the existence of the solutions is obtained only for small flux. In such case, the velocity converges exponentially to a slightly distorted Poiseuille velocity at far field. 

For the general pipe-like domain, the solvability of problem  \eqref{NS}-\eqref{BC} and \eqref{flux constraint} was firstly studied in \cite{LS0,SP,S2}. Given nonzero flux, the existence of solutions with finite Dirichlet integral can be obtained only in the pipes with wide outlets.  In \cite{LS}, Ladyzhenskaya and Solonnikov considered the pipes with both narrow and wide outlets and  proved that the problem is solvable for any flux, provided the outlets of the pipe satisfy certain geometric assumptions. The solutions have either finite or infinite Dirichlet integral over the outlets dependent of the shape of the outlets. One may refer to \cite{KP0,KP,MF1,MF2,P0,S1,S3} for more results on the well-posedness of the incompressible Navier-Stokes system in pipe-like domains. 

The asymptotic behavior of the solutions in domains with noncompact boundary  was studied in \cite{AF,Ga,NP,P1,P2,P3} and references therein. It is believed that the behavior of solutions to Navier-Stokes problem at far field strongly depends on the geometry of outlets to infinity. Suppose that the outlet is characterized by 
\[\Omega\cap \{x:~x_1>0\}=\left\{x:~x_1>0,~\sqrt{x_2^2+\cdots+ x_N^2}<g(x_1)\right\}\]
for $N=2,3$. If $g(t)=Ct^{1-\alpha}$ for some constants $C$ and $\alpha\in(0,1)$, then the explicit asymptotic expansion for the solution is constructed in \cite{NP}. If $g(t)$ satisfies the global Lipschitz condition and $g'(t)\to 0$ as $t\to \infty$, the pointwise decay of the solutions with arbitrary flux is obtained in \cite{P1} for three-dimensional case. However, for two-dimensional channels, the results are proved only for small flux.

In this paper, we study the problem \eqref{NS}-\eqref{BC} and \eqref{flux constraint} in a two-dimensional channel $\Omega$ of the form
\begin{equation}\label{defOmega}
\Omega=\{x=(x_1,x_2):x_1\in \mathbb{R},~f_1(x_1)<x_2<f_2(x_1)\},
\end{equation}
where $f_1$ and $f_2$ are assumed to be smooth functions. 
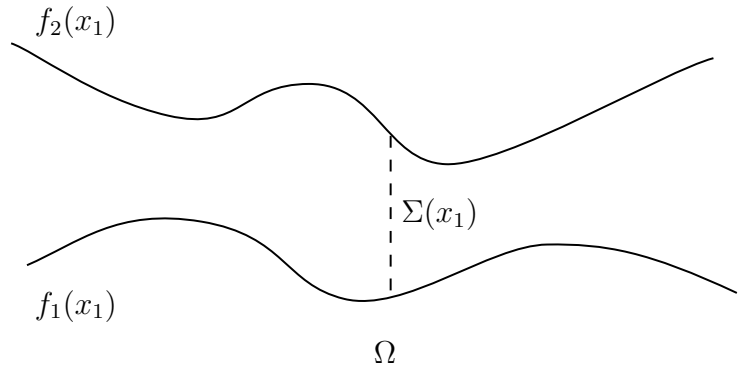
\begin{figure}[h]
	\centering

	\tikzset{every picture/.style={line width=0.75pt}} 

	\begin{tikzpicture}[x=0.75pt,y=0.75pt,yscale=-1,xscale=1]
	
	\draw    (118.6,84.5) .. controls (130.6,88.5) and (160.6,112.5) .. (197.8,121) .. controls (235,129.5) and (234.12,105.1) .. (268.8,105) .. controls (303.48,104.9) and (307.37,144.08) .. (337.6,145.5) .. controls (367.83,146.92) and (447.45,98.71) .. (472.8,92) ;
	\draw    (126.6,196.5) .. controls (152.46,185.3) and (171.99,167.36) .. (214.6,174.5) .. controls (257.21,181.64) and (256.17,206.04) .. (286.6,213.5) .. controls (317.03,220.96) and (359.6,186.5) .. (388.95,186.06) .. controls (418.31,185.62) and (439.6,189.5) .. (484.6,210.5) ;
	\draw  [dash pattern={on 4.5pt off 4.5pt}]  (310,131.2) -- (310,212.2) ;
	
	\draw (300,233.4) node [anchor=north west][inner sep=0.75pt]    {$\Omega $};
	\draw (129,64.4) node [anchor=north west][inner sep=0.75pt]    {$f_{2}( x_{1})$};
	\draw (129,207.4) node [anchor=north west][inner sep=0.75pt]    {$f_{1}( x_{1})$};
	\draw (314,161.4) node [anchor=north west][inner sep=0.75pt]    {$\Sigma ( x_{1})$};

	\end{tikzpicture}

	\caption{The channel $\Omega$}
\end{figure}

Before stating the main results of this paper, the definitions of some function spaces and the weak solution are introduced.
\begin{definition}\label{def1}
	Given a bounded domain $D\subset  \mathbb{R}^2$, denote
	\[
	L_0^2(D)=\left\{w(x): w\in L^2(D), \, \int_D w(x)dx =0 \right\}.
	\]
Given $\Omega$ defined in \eqref{defOmega},		for any constants $ a<b$ and $T>0$, denote
	\begin{equation*}
	\Omega_{a,b}=\{(x_1,x_2)\in \Omega:a<x_1<b\}\  \text{ and }\ \Omega_T=\Omega_{-T,T}.
	\end{equation*}
Define
\begin{equation*}
	H_{0,\sigma}^1(\Omega_{a,b})=\{\Bu \in H_0^1(\Omega_{a,b}):~\operatorname{div} \Bu=0\text{ in }\Omega_{a,b}\}
\end{equation*}
and 
\begin{equation*}
	H_\sigma(\Omega)=\left\{\Bu\in H_{loc}^1(\Omega):~{\rm div}\,\Bu=0 \text{ in }\Omega,~\Bu=0 \text{ on } \partial\Omega\right\}.
\end{equation*}
\end{definition}

\begin{definition}
	A vector field $\Bu\in H_\sigma(\Omega)$ is said to be a weak solution of the Navier-Stokes system \eqref{NS} with Dirichlet boundary conditions \eqref{BC} if for any $T>0$, $\Bu$ satisfies
\begin{equation}\label{weak solution}
\int_{\Omega}\nabla\Bu:\nabla\Bp+\Bu\cdot\nabla\Bu\cdot \Bp\,dx=0 \quad \text{for any}\,\,  \Bp \in H_{0,\sigma}^1(\Omega_T). 	
\end{equation}
\end{definition}

Denote
\begin{equation}\label{deffbar}
	f(x_1):= f_2(x_1)-f_1(x_1)\quad \text{and}\quad \bar{f}(x_1):= \frac{f_2(x_1)+f_1(x_1)}{2}.
\end{equation}
In this paper, we always assume that
\begin{equation}\label{assumpf}
	\inf_{x_1\in \mathbb{R}} f(x_1)=\underline{d} >0\quad \text{and}\quad \max_{i=1,2} \|f_i'\|_{C(\mathbb{R})}=\beta <+\infty.
\end{equation}

 The first main result of this paper can be stated as follows.
\begin{theorem}\label{bounded channel}
	Let $\Omega$ be the domain given in \eqref{defOmega} . If the width of the channel $\Omega$  is uniformly bounded, i.e.,
	\be \label{bounded}
	f(x_1) := f_2(x_1)-f_1(x_1) \leq \overline{d} < +\infty, \ \ \ \text{for any}\, \, x_1 \in \mathbb{R},
	\ee
	then the problem \eqref{NS}-\eqref{BC} and \eqref{flux constraint} has a solution $\Bu\in H_\sigma(\Omega)$ satisfying the estimate
	\begin{equation}\label{1-4}
		\|\nabla \Bu\|_{L^2(\Omega_t)}^2\leq \overline{C}(1+t),\ \ \text{ for any }t\ge 0,
	\end{equation}
	where $\overline{C}$ is a positive constant independent of $t$.
	Furthermore, the solution satisfies the following properties.
	\begin{enumerate}
		\item[(a)] There exists a constant $\Phi_0>0$ such that for any flux $\Phi\in [0,\Phi_0)$, the solution $\Bu$ obtained in Theorem \ref{bounded channel} is unique
	in the class of functions satisfying \eqref{1-4}.

		\item[(b)]	If, in addition, the outlet of the channel is straight, i.e., there exists a constant $k>0$ such that $\Omega\cap \{x_1>k\}=\{(x_1,x_2): x_1>k, x_2\in (c_1, c_2)\}$ for some constants $c_1<c_2$, then there exists a constant $\Phi_1>0$ such that for any flux $\Phi\in [0,\Phi_1)$, the solution $\Bu$ tends to the corresponding Poiseuille flow $\BU=U(x_2)\Be_1$ in the sense
		\begin{equation*}
			\|\Bu-\BU\|_{H^1(\Omega\cap \{x_1>k\})}<\infty.
		\end{equation*}
	\end{enumerate}
\end{theorem}

\begin{remark}
	The constant $\overline{C}$ depends only on the flux $\Phi$, and the domain $\Omega$. More precisely, it depends on $\Phi$, $\|f_i\|_{C^2(\R)}$, and $\underline{d}$.
\end{remark}

For the flows in channels with unbounded outlets, we have also the following theorem.
\begin{theorem}\label{unbounded channel}
	Let $\Omega$ be the domain given in \eqref{defOmega} . Suppose that	
	\begin{equation}\label{assumpf''}
		\max_{i=1,2}\sup_{x_1\in \R}|(f_i^{\prime \prime} f)(x_1)|<\infty.
	\end{equation}
\begin{enumerate}
	\item[(\romannumeral1)] (Existence) The problem \eqref{NS}-\eqref{BC} and \eqref{flux constraint} has a solution $\Bu\in H_\sigma(\Omega)$  satisfying the estimate
	\begin{equation}\label{1-9}
		\|\nabla \Bu\|_{L^2(\Omega_t)}^2\leq \widetilde{C}\left(1+\int_{-t}^t f^{-3}(x_1)\,dx_1 \right)\text{ for any }t\ge 0,
	\end{equation}
	where $\widetilde{C}$ is a positive constant depending only on the flux $\Phi$, and $\Omega$.
	
	\item[(\romannumeral2)] (Uniqueness)
	 If, in addition, it holds  that either
		\begin{equation}\label{1-16}
			\left|\int_0^{\pm\infty}f^{-3}(\tau)\,d\tau\right|=\infty,\ \ \ \lim_{|t|\to +\infty} f'(t)= 0,
		\end{equation}
		or
		\begin{equation}\label{1-17}
			\left|\int_0^{\pm\infty}f^{-3}(\tau)\,d\tau\right|<\infty,\ \ \ \lim_{t\to \pm\infty}\frac{\sup_{\pm \tau \ge |t|}f'(\tau)}{\left|\int_t^{\pm \infty} f^{-3}(\tau)\,d\tau\right|^\frac12}= 0,
		\end{equation}
	 then there exists a constant $\Phi_2>0$ such that for any flux $\Phi\in [0,\Phi_2)$, the solution $\Bu$ obtained in Theorem \ref{unbounded channel} is unique in the class of functions satisfying \eqref{1-9}.
\end{enumerate}
\end{theorem}

There are some remarks in order.
\begin{remark}
If $f(t)$ is a power function at far field, the conditions \eqref{1-16}-\eqref{1-17} are equivalent to
\begin{equation*}
	f(t)=o(t^\frac35)\ \ \ \ \text{ as }|t|\to\infty.
\end{equation*}
\end{remark}

\begin{remark}
It should be emphasized that there is no restriction on the flux $\Phi$ for the existence of solutions in both Theorems \ref{bounded channel} and \ref{unbounded channel}.
\end{remark}

\begin{remark} In certain sense, the estimates \eqref{1-4} and \eqref{1-9} are optimal, as there exists a constant $C>0$ such that
\begin{equation}\label{optimalest}
\Phi^2  \int_{-t}^t f^{-3} (x_1) \, dx_1 \leq C\|\nabla \Bu\|_{L^2(\Omega_t)}^2.
\end{equation}
 Indeed, for any flow $\Bu$ with flux $\Phi$, one uses Poincar\'{e}'s inequality to obtain  
\begin{equation*}\begin{aligned}
		\Phi^2= \left|\int_{\Sigma(x_1)}\Bu\cdot \Bn\,ds\right|^2\leq |\Sigma(x_1)|\int_{\Sigma(x_1)}|\Bu|^2\,dx_2  \leq C|f (x_1)|^3\int_{\Sigma(x_1)}|\nabla \Bu|^2\,dx_2.
\end{aligned}\end{equation*}
Integrating this inequality with respect to $x_1$ over $(-t,t)$, one has \eqref{optimalest}.
\end{remark}

Furthermore, we can obtain the pointwise decay rate of the solution $\Bu$ obtained in Theorem \ref{unbounded channel}.
\begin{theorem}\label{decay rate}
	Let $\Bu=\Bv+\Bg$ be the solution obtained in Theorem \ref{unbounded channel}. Assume further that either \eqref{1-16} or \eqref{1-17} holds. Then one has 
	\[|\Bu(x)|\leq \frac{C}{f(x_1)} \text{ for any }x\in \Omega,\]
	where $C$ is a constant depending only on $\Omega$ and $\Phi$.
\end{theorem}

\section{Preliminaries}\label{secpreliminary}
In this section, some elementary but important lemmas are collected. We first give the Poincar\'e type inequality and Sobolev embedding inequality in channels, which can be proved in a similar way to that in \cite[Section 2.]{SWX22}.
\begin{lemma}\label{lemmaA1}
For any $\Bv \in H^1(\Omega_{a,b})$ satisfying $\Bv=0$ on the boundary $\partial\Omega\cap \partial\Omega_{a,b}$, one has
\begin{equation}\label{A1-0}
	\left\|\frac{v_1}{f}\right\|_{L^2(\Omega_{a,b})}\leq M_0 \left\|\partial_{x_2} v_1\right\|_{L^2(\Omega_{a,b})}
\end{equation}
and
\begin{equation}
\left\|\Bv\right\|_{L^2(\Omega_{a,b})}\leq M_1(\Omega_{a,b}) \left\|\nabla\Bv\right\|_{L^2(\Omega_{a,b})},
\end{equation}
 where $M_0$ is a uniform constant independent of the domain $\Omega_{a,b}$ and
 \begin{equation}\label{defM1}
 M_1(\Omega_{a,b})=C\|f\|_{L^\infty(a, b)}
 \end{equation}
 with a universal constant $C$.
\end{lemma}

\begin{lemma}\label{lemmaA2} For any $\Bv \in H^1(\Omega_{a,b})$ satisfying $\Bv=0$ on $\partial\Omega_{a,b}\cap \partial\Omega$, one has
\begin{equation*}
\|\Bv\|_{L^4(\Omega_{a,b})}\leq M_4(\Omega_{a,b}) \| \nabla \Bv\|_{L^2(\Omega_{a,b})},
\end{equation*}
where
\begin{equation}\label{defM4}
M_4(\Omega_{a,b})=C\left[ (b-a)^{-1} M _1(\Omega_{a,b})+1 \right]^\frac12|\Omega_{a,b}|^\frac14
\end{equation}
with a universal constant $C$ and $M_1=M_1(\Omega_{a,b})$ defined in \eqref{defM1}.
\end{lemma}

The following lemma on the solvability of the divergence equation is used to obtain the estimates involving pressure. For the proof, one may refer to \cite[Theorem  \uppercase\expandafter{\romannumeral3}.3.1 ]{Ga} and \cite{Bo}.
\begin{lemma}\label{lemmaA5}
Let $D \subset \R^n$ be a bounded Lipschitz domain. Then there exists a constant $M_5$ such that for any $w\in L_0^2(D)$, the  problem
\begin{equation}\label{A5-1}
\left\{\begin{aligned}
{\rm div}~\Ba=w ~~~~~~~~~~&\text{ in }D,\\
\Ba=0 ~~~~~~~~~~~~~~~~~&\text{ on }\partial D
\end{aligned}\right.
\end{equation}
has a solution $\Ba \in H^1_0(D)$ satisfying
\[
\|\nabla\Ba\|_{L^2(D)}\leq M_5(D)\|w\|_{L^2(D)}.
\]  
In particular, if the domain is of the form
\begin{equation*}
 D=\bigcup_{k=1}^N D_k,
\end{equation*}
where each $D_k$ is star-like with respect to some open ball $B_k$ with  $\overline{B_k}\subset D_k$, then the constant $M_5(D)$ admits the following estimate
\begin{equation}\label{A5-2}M_5(D)\leq C_D \left(\frac{R_0}{R}\right)^n\left(1+\frac{R_0}{R}\right).
\end{equation}
Here, $R_0$ is the diameter of the domain $D$, $R$ is the smallest radius of the balls $B_k$, and
\begin{equation}\label{A5-3}
C_D=\max_{1\leq k\leq N}\left(1+\frac{|D_k|^\frac12}{|\tilde{D}_k|^\frac12}\right)\prod_{i=1}^{k-1}\left(1+\frac{|\hat{D}_i\setminus D_i|^\frac12}{|\tilde{D}_i|^\frac12}\right),
\end{equation}
with $\tilde{D}_i=D_i\cap \hat{D}_i$ and $\hat{D}_i=\bigcup_{j=i+1}^N D_j$.
\end{lemma}

We next recall the estimates for some differential inequalities, whose proof can be found in \cite{LS}. These differential inequalities play crucial role in the estimates for local Dirichlet norm.
\begin{lemma}\label{lemmaA4}~
(1) Let $z(t)$ and $\varphi(t)$ be the nontrivial, nondecreasing, and nonnegative smooth functions. Suppose that $\Psi(t, s)$ is a monotonically increasing function with respect to $s$, equals to zero for $s=0$ and tends to $\infty$ as $s\to \infty$. Suppose that $\delta_1\in (0,1)$ is a fixed constant and for any  $t\in [t_0,T]$, $z(t)$ and $\varphi(t)$ satisfy
\begin{equation}\label{A4-1}
z(t)\leq \Psi(t, z'(t))+(1-\delta_1)\varphi(t)
\end{equation}
and
\begin{equation}\label{A4-2}
\varphi(t)\geq \delta_1^{-1}\Psi(t, \varphi'(t)).
\end{equation}
If $z(T) \leq \varphi(T)$, then
\begin{equation}\label{A4-3}
z(t)\leq \varphi(t)\text{ for any }t\in [t_0,T].
\end{equation}
(2) Assume that $\Psi(t, s) = \Psi(s)$ and the inequalities \eqref{A4-1} and \eqref{A4-2} are fulfilled for any $t\ge t_0$.  If
\begin{equation*}
  \liminf_{t\to \infty} \frac{z(t)}{\varphi(t)}<1\ \ \ \ \text{ or } \ \ \ \ \lim_{t\to \infty}\frac{z(t)}{\widetilde{z}(t)}=0
\end{equation*}
where $\widetilde{z}(t)$ is the positive solutions to the equation
\begin{equation*}
  \tilde{z}(t)=\delta_1^{-1}\Psi(\tilde{z}'(t)),
\end{equation*}
then \eqref{A4-3}
holds.

(3) Assume that $\Psi(t, s)= \Psi(s)$ and the function $z(t)$ is nontrivial and nonnegative. If there exist $m>1,t_0, s_1\ge 0,c_0>0$ such that
\begin{equation*}
  z(t)\leq \Psi(z'(t))~~~~\text{ for any } t\ge t_0
\end{equation*}
  and
\begin{equation*}
\Psi(s)\leq c_0 s^m \text{ for any }s \ge s_1,
\end{equation*}
then
\begin{equation*}
  \liminf_{t\to \infty} t^\frac{-m}{m-1}z(t)>0.
\end{equation*}
\end{lemma}

\section{Flux carrier and the approximate problem}\label{secflux}
In this section, we construct the so-called flux carrier, which is a solenoidal vector field with flux $\Phi$ and satisfies no-slip boundary condition \eqref{BC}, and study the well-posedness of the approximate problem on bounded domain $\Omega_{a,b}$. 

In fact, the flux carrier is given in \cite{SWX22}, we nevertheless give the construction for the completeness.  Let $\mu(t)$ be a smooth function on $\R$ which satisfies
\begin{equation*}
	\mu(t)=\left\{
	\begin{aligned}
		&0 \,\,\,\,\,\,\text{ if } t\ge 1,\\
		&1 \,\,\,\,\,\,\text{ if } t\le 0.
	\end{aligned}
	\right.
\end{equation*}
For any $\varepsilon\in(0,1)$ to be determined, define
\begin{equation}\label{defg}
\Bg=(g_1,g_2)=(\partial_{x_2}G,-\partial_{x_1}G),
\end{equation}
where
\begin{equation*}
G(x_1,x_2;\varepsilon)=\left\{
\begin{aligned}&\Phi\mu\left(1+\varepsilon \ln \frac{f_2(x_1)-x_2}{x_2-\bar{f}(x_1)}\right),&\text{ if }x_2>\bar{f}(x_1),\\
&	0,&\text{ if }x_2\leq \bar{f}(x_1),
\end{aligned}\right.
\end{equation*}
with $\bar{f}$ defined in \eqref{deffbar}. Clearly, $\Bg$ is a smooth solenoidal vector field.

Noting
\begin{equation*}
G(x_1,x_2;\varepsilon)=\left\{
\begin{aligned}
&\Phi, &&\text{ if }x_2\text{ near }f_2(x_1),\\
&0, &&\text{ if }x_2\leq \bar{f}(x_1),
\end{aligned}
\right.
\end{equation*}
one can see that the vector field $\Bg$ vanishes near the boundary $\partial \Omega$ and satisfies the flux constraint \eqref{flux constraint}. Since $\operatorname{supp} \mu' \subset [0,1]$, one has
\begin{equation*}
\operatorname{supp} \Bg\subset \left\{ (x_1, x_2) \in \Omega:~e^{-\frac1\varepsilon}\leq \frac{f_2(x_1)-x_2}{x_2-\bar{f}(x_1)}\leq 1 \right\}.
\end{equation*}
This implies that for any $x\in \operatorname{supp}\Bg$, one has
\begin{equation}\label{1-6}
 f_2(x_1)-x_2 \leq  x_2-\bar{f}(x_1)\leq e^\frac1\varepsilon \left(f_2(x_1)-x_2\right).
\end{equation}
It also follows from \eqref{1-6} that for any $x\in \operatorname{supp}\Bg$, one has
\begin{equation*}
	2(x_2-\bar{f}(x_1))\ge f_2(x_1)-x_2+x_2-\bar{f}(x_1)=f_2(x_1)-\bar{f}(x_1)=\frac{f(x_1)}{2}
\end{equation*}
and
\begin{equation*}
	(1+e^{-\frac1\varepsilon})(x_2-\bar{f}(x_1))\leq  x_2-\bar{f}(x_1)+f_2(x_1)-x_2=f_2(x_1)-\bar{f}(x_1)=\frac{f(x_1)}{2},
\end{equation*}
where $f$ is defined in \eqref{deffbar}.
Hence, one has
\begin{equation}\label{1-7}
	\frac{f(x_1)}{4}\leq x_2-\bar{f}(x_1)\leq \frac{1}{1+e^{-\frac1\varepsilon}}\frac{f(x_1) }{2}\leq \frac{f(x_1)}{2}
\end{equation}
and
\begin{equation}\label{1-10}
	f_2(x_1)-x_2\ge e^{-\frac1\varepsilon }(x_2-\bar{f}(x_1)) \ge e^{-\frac1\varepsilon }\frac{f(x_1)}{4}.
   \end{equation}

Moreover, straightforward computations give
\begin{equation}\label{1-11-1}
\begin{aligned}
	g_1=&\Phi\partial_{x_2} \mu\left(1+\varepsilon \ln (f_2 (x_1)-x_2)- \varepsilon \ln(x_2-\bar{f}(x_1))\right)\\
	=&\varepsilon \Phi \mu'(\cdot) \left(\frac{-1}{f_2 (x_1)-x_2}-\frac{1}{x_2-\bar{f}(x_1)}\right)
\end{aligned}
\end{equation}
and
\begin{equation}\label{1-11-2}
\begin{aligned}
	g_2=&-\Phi\partial_{x_1} \mu\left(1+\varepsilon \ln (f_2 (x_1)-x_2)- \varepsilon \ln\left(x_2-\bar{f}(x_1)\right)\right)\\
=&-\varepsilon \Phi \mu'(\cdot) \left(\frac{f_2'(x_1)}{f_2 (x_1)-x_2}+\frac{\bar{f}'(x_1)}{x_2-\bar{f}(x_1)}\right),
\end{aligned}
\end{equation}
where $\mu'(\cdot)=\mu'\left(1+\varepsilon \ln (f_2(x_1)-x_2)- \varepsilon \ln \left(x_2-\bar{f}(x_1) \right)\right)$.

The following lemma collects some properties of the flux carrier $\Bg$, which play an important role in the construction of approximate solutions, especially when $\Phi$ is not small. One may refer to \cite{SWX22} for the detail of the proof.
\begin{lemma}\label{lemma1}
For any function $w\in H^1(\Omega_{a,b})$ satisfying $w=0$ on the boundary $\partial\Omega\cap \partial\Omega_{a,b}$, it holds that
\begin{equation*}
\int_{\Omega_{a,b}} \Bg^2w^2\,dx\leq C \Phi^2\varepsilon^2\int_{\Omega_{a,b}}|\partial_{x_2} w |^2\,dx.
\end{equation*}
 Moreover, if $f_i(i=1,2)$ satisfies \eqref{assumpf''}, then one has
\begin{equation*}
	|\Bg|\leq \frac{C(\varepsilon) \Phi}{f(x_1)},\ \ \ \ \ |\nabla\Bg|\leq \frac{C(\varepsilon, \gamma) \Phi}{f^2(x_1)},
\end{equation*}
and
\begin{equation*}
  \int_{\Omega_{a,b}}|\nabla \Bg|^2+|\Bg|^4\,dx\leq C(\epsilon, \gamma) ( \Phi^2 +  \Phi^4) \int_{a}^b f^{-3}(x_1)\,dx_1,
\end{equation*}
where $C(\varepsilon)$ is a constant depending on $\varepsilon$ and $C(\varepsilon,\gamma)$ depends on $\varepsilon$ and $\gamma$ with
\begin{equation}
\gamma=	\max_{i=1,2}\sup_{x_1\in \mathbb{R}} |f_i''(x_1)f(x_1)|.
\end{equation}

\end{lemma}

Given the flux carrier $\Bg$ constructed in \eqref{defg}, if $\Bu$ satisfies \eqref{NS}-\eqref{BC} and \eqref{flux constraint}, then $\Bv=\Bu-\Bg$ satisfies
\begin{equation}\label{NS1}
\left\{
\begin{aligned}
&-\Delta \Bv+\Bv\cdot \nabla \Bg +\Bg\cdot \nabla \Bv+\Bv\cdot \nabla \Bv  +\nabla p=\Delta \Bg-\Bg\cdot \nabla\Bg   ~~~~&\text{ in }\Omega,\\
&{\rm div}~\Bv=0&\text{ in }\Omega,\\
&\Bv=0&\text{ on }\partial\Omega,\\
&\int_{\Sigma(x_1)} \Bv\cdot \Bn \,ds=0&\text{ for any }x_1\in \mathbb{R}.
\end{aligned}\right.
\end{equation}

The weak solutions of \eqref{NS1} is defined as follows.
\begin{definition}
A vector field $\Bv\in H_\sigma(\Omega)$ is said to be a weak solution of the problem \eqref{NS1} if for any $\Bp\in H_{0,\sigma}^1(\Omega_T)$ with $T>0$, one has
\begin{equation*}
\int_{\Omega}\nabla\Bv:\nabla\Bp+(\Bv\cdot \nabla \Bg +(\Bg+\Bv)\cdot\nabla \Bv)\cdot \Bp\,dx
=\int_{\Omega}(\Delta \Bg-\Bg\cdot \nabla \Bg)\cdot \Bp\,dx.
\end{equation*}
\end{definition}

In the rest of this section, we study the following approximate problems of \eqref{NS1} on the bounded domain $\Omega_{a,b}$,
\begin{equation}\label{aNS}
\left\{
\begin{aligned}
&-\Delta \Bv+\Bv\cdot \nabla \Bg +\Bg\cdot \nabla \Bv+\Bv\cdot \nabla \Bv  +\nabla p=\Delta \Bg-\Bg\cdot \nabla \Bg   ~~~~&\text{ in }\Omega_{a,b},\\
&{\rm div}~\Bv=0&\text{ in }\Omega_{a,b},\\
&\Bv=0&\text{ on }\partial\Omega_{a,b}
\end{aligned}\right.
\end{equation}
and its linearized problem
\begin{equation}\label{laNS}
\left\{\begin{aligned}
&-\Delta \Bv+\Bv\cdot \nabla \Bg +\Bg\cdot \nabla \Bv +\nabla p=\Delta \Bg-\Bg\cdot \nabla \Bg   ~~~~&\text{ in }\Omega_{a,b},\\
&{\rm div}~\Bv=0&\text{ in }\Omega_{a,b},\\
&\Bv=0&\text{ on }\partial\Omega_{a,b}.\end{aligned}\right.\end{equation}

The weak solutions of problems \eqref{aNS} and \eqref{laNS} can be defined as follows.
\begin{definition}
A vector field $\Bv\in H_{0,\sigma}^1(\Omega_{a,b})$ is a weak solution of the problem \eqref{aNS} and \eqref{laNS}, respectively
if for any  $\Bp\in H_{0,\sigma}^1(\Omega_{a,b})$, $\Bv$ satisfies
\begin{equation}\label{2-1}
\begin{aligned}
&\int_{\Omega_{a,b}}\nabla\Bv:\nabla\Bp+(\Bv\cdot \nabla \Bg +(\Bg+\Bv)\cdot \nabla \Bv)\cdot \Bp\,dx 
=\int_{\Omega_{a,b}}(\Delta \Bg-\Bg\cdot \nabla \Bg)\cdot \Bp\,dx
\end{aligned}
\end{equation}
and
\begin{equation}\label{2-2}\begin{aligned}
&\int_{\Omega_{a,b}}\nabla\Bv:\nabla\Bp+(\Bv\cdot \nabla \Bg +\Bg\cdot \nabla \Bv)\cdot \Bp\,dx
=\int_{\Omega_{a,b}}(\Delta \Bg-\Bg\cdot \nabla \Bg)\cdot \Bp\,dx,
\end{aligned}\end{equation}respectively.
\end{definition}

Next, we use Leray-Schauder fixed point theorem (cf. \cite[Theorem 11.3]{GT}) to prove the existence of solutions to the approximate  problem \eqref{aNS}. To this end, the well-posedness of the linearized problem \eqref{laNS} is first established by the following lemma.
\begin{lemma}\label{lemma5}
For any $\Bh\in L^\frac43(\Omega_{a,b})$, there exists a unique solution $\Bv\in H_{0,\sigma}^1(\Omega_{a,b})$ such that for any $\Bp\in H_{0,\sigma}^1(\Omega_{a,b})$, it holds that
\begin{equation}\label{2-12}
\int_{\Omega_{a,b}}\nabla\Bv:\nabla\Bp+(\Bv\cdot \nabla \Bg +\Bg\cdot \nabla \Bv)\cdot \Bp\,dx
=\int_{\Omega_{a,b}}\Bh\cdot\Bp\,dx.
\end{equation}
\end{lemma}
Here we omit the proof of Lemma \ref{lemma5}, which is based on Lax-Milgram theorem and can be found in \cite{SWX22}. 
Note that $\Delta \Bg-\Bg\cdot \nabla \Bg \in L^\frac43(\Omega_{a,b})$ since $\Bg\in C^2(\bar{\Omega})$. Therefore, the existence of solutions to the linearized problem \eqref{laNS} is a consequence of Lemma \ref{lemma5}.
\begin{cor} For any $a<b$, the linearized problem \eqref{laNS} admits a unique solution $\Bv\in H_{0,\sigma}^1(\Omega_{a,b})$.
\end{cor}
 
Finally, the existence of solutions for the approximate problem \eqref{aNS} follows from the application of Leray-Schauder fixed point theorem. 

\begin{pro}\label{appro-existence}
For any $a<b$, the problem \eqref{aNS} has a weak solution $\Bv\in H_{0,\sigma}(\Omega_{a,b})$ satisfying
\begin{equation}\label{2-10-1}
\|\nabla \Bv\|_{L^2(\Omega_{a,b})}^2\leq C_0\int_{\Omega_{a,b}} |\nabla \Bg|^2+|\Bg|^4\,dx,
\end{equation}
where the constant $C_0$ is independent of $a$ and $b$.
\end{pro}

As long as the existence of weak solution $\Bv$ for the problem \eqref{2-1} is established, one can further obtain the associated pressure for \eqref{aNS} with the aid of the following lemma,  whose proof can be found in \cite[Theorem \uppercase\expandafter{\romannumeral3}.5.3]{Ga}.
\begin{pro}\label{pressure}
The vector field $\Bv\in H_{0,\sigma}^1(\Omega_{a,b})$ is a weak solution of the approximate problem \eqref{aNS} if and only if there exists a function $p\in L^2(\Omega_{a,b})$ such that the identity
\begin{equation}\label{2-11}
\begin{aligned}
\int_{\Omega_{a,b}}\nabla\Bv:\nabla\Bp+(\Bv\cdot \nabla \Bg +(\Bg+\Bv)\cdot \nabla \Bv)\cdot \Bp\,dx&\\
-\int_{\Omega_{a,b}}p{\rm div}\Bp\,dx&=\int_{\Omega_{a,b}}(\Delta \Bg-\Bg\cdot \nabla \Bg)\cdot \Bp\,dx
\end{aligned}
\end{equation}
holds for any $\Bp\in H_0^1(\Omega_{a,b})$.
\end{pro}

\section{Flows in channels with bounded outlets}\label{secexist1}
In this section, we investigate the flows in channels with bounded outlets. Using the technique developed in \cite{LS}, one can prove the existence of solutions to the problem \eqref{NS}-\eqref{BC} and \eqref{flux constraint} by showing a uniform estimate for the approximate solutions obtained in Proposition \ref{appro-existence}. Since the three-dimensional problem is already solved in \cite{LS}, Theorem \ref{bounded channel} can be proved in a similar way with the aid of Lemmas \ref{lemmaA1}-\ref{lemmaA5}. Here we give a sketch of the proof. One may refer to \cite{LS,SWX22} for the detail.


\begin{lemma}\label{lemma3}
Assume that $\Omega$ is a channel with bounded outlets, i.e., $f$ satisfies \eqref{assumpf}. Let $\Bv^T$ be the solution of the approximate problem \eqref{aNS} in $\Omega_T$, which is obtained in Proposition \ref{appro-existence}. Then one has
\begin{equation}\label{3-0}
\|\nabla\Bv^T\|_{L^2(\Omega_t)}^2 \leq C_3+C_4t \ \ \ \text{ for any }1< t\leq T-1,
\end{equation}
where the constants $C_3$ and $C_4$ are independent of $t$ and $T$.
\end{lemma}

According to Lemma \ref{lemma3}, for any $1\leq t\le T-1$,
\begin{equation*}
\|\nabla \Bv^T\|_{L^2(\Omega_{t})}^2\leq C_3 +C_4 t.
\end{equation*}
 Since the constants $C_3$ and $C_4$ are independent of $t$ and $T$, one can extend $\Bv^T$ by zero to the whole channel $\Omega$ and take the limit $T\to \infty$ and select a subsequence which converges weakly in $H^1_{loc}(\Omega)$ to a solution $\Bv$ of \eqref{NS1}. Moreover, $\Bv$  satisfies the estimate
\begin{equation*}
\|\nabla \Bv\|_{L^2(\Omega_{t})}^2\leq C_3 +C_4 t.
\end{equation*}
With the estimate for $\Bg$ in Lemma \ref{lemma1}, one has the following proposition on the existence of solutions.
\begin{pro}\label{straight-existence}
The problem \eqref{NS}-\eqref{BC} and \eqref{flux constraint} has a solution  $\Bu=\Bv+\Bg\in H_\sigma(\Omega)$ satisfying
\begin{equation}\label{3-27}
\|\nabla \Bu\|_{L^2(\Omega_{t})}^2 \leq \tilde{C}(1+t),
\end{equation}
where the constant $\tilde{C}$ depends only on $\Phi$ and $\Omega$.
\end{pro}
\begin{remark}
	There exists a constant $C>0$ such that
	for any fixed subdomain $\Omega_{a,b}$, if $\Phi>0$ is sufficiently small, one has
	\[
	\int_{\Omega_{a,b}}|\nabla \Bg|^2+|\Bg|^4\,dx\leq C(b-a)\Phi^2.
	\]
	Therefore, there exists a $\Phi_0>0$ such that if $\Phi\in [0, \Phi_0)$, one has
\[
C_3+C_4+\tilde{C}\leq C\Phi^2,
\]
where $C_3,C_4$, and $\tilde{C}$ are the constant appeared in \eqref{3-0} and \eqref{3-27}.
\end{remark}

Similar to Lemma \ref{pressure}, one can also define the pressure of the problem \eqref{NS} and \eqref{BC}.
\begin{pro}\label{pressure1}
The vector field $\Bu\in H_\sigma(\Omega)$ is a weak solution of the problem \eqref{NS} and \eqref{BC} if and only if there exists a function $p\in L^2_{loc}(\Omega)$ such that  for any $\Bp\in H_0^1(\Omega_T)$ with $T>0$, it holds that
\begin{equation}\label{3-14-1}
\int_{\Omega_T}\nabla\Bu:\nabla\Bp+\Bu\cdot \nabla \Bu \cdot \Bp - p{\rm div}\Bp\,dx=0.
\end{equation}

\end{pro}

Actually, one can show that the Dirichlet norm of the solution $\Bu$ is uniformly bounded in any subdomain $\Omega_{t-1,t}$.
\begin{pro}\label{uniform estimate}
Let $\Bu$ be the solution obtained in Proposition \ref{straight-existence}. Then there exists a constant $C_7$ such that
\begin{equation}\label{estv4.4}
\|\nabla \Bu\|_{L^2(\Omega_{t-1,t})}^2\leq C_7  \quad \text{for any}\,\, t\in \mathbb{R}.
\end{equation}
\end{pro}

With the help of the uniform estimate given in Proposition \ref{uniform estimate}, we can prove the uniqueness of the solution when the flux is sufficiently small.
\begin{pro}\label{uniqueness}
There exists a constant $\Phi_0>0$ such that for any flux $\Phi\in [0,\Phi_0)$, the solution obtained in Proposition \ref{straight-existence} is unique.
\end{pro}

In particular, if an outlet of the channel is straight, for example,
\begin{equation*}
\Sigma(x_1)=\Sigma^\sharp(x_1):=(-1,1)\ \ \ \text{ when $x_1>0$},
\end{equation*}
we shall show that the solution obtained in Proposition \ref{straight-existence} tends to Poiseuille flow $\BU=U(x_2)\Be_1=\frac{3\Phi}{2}(1- x_2^2)\Be_1$ at infinity, where $\BU$ is the  solution of the Navier-Stokes system with Dirichlet boundary condition in the straight channel ${\Omega}^\sharp=\{(x_1,x_2):~x_1\in \R,~x_2\in (-1,1)\}$.
\begin{pro}\label{Poiseuille}
Assume that the outlet $\Omega^+=\{ x\in\Omega:~x_1>0\}= (0, +\infty) \times (-1, 1)$ is straight. There exists a constant $\Phi_1>0$, such that if  $\Phi\in [0,\Phi_1)$,  and the solution $\Bu$ of the problem \eqref{NS}-\eqref{BC} and \eqref{flux constraint}  satisfies
\begin{equation}\label{3-26}
\liminf_{t\rightarrow + \infty} t^{-3}\int_{\Omega^+_t} |\nabla \Bu|^2\,dx =0,
\end{equation}
where $\Omega^+_t=\{ x\in\Omega:~0<x_1<t\}$, then it holds that
\begin{equation*}
\|\Bu-\BU\|_{H^1(\Omega^+)}<\infty.
\end{equation*}
\end{pro}
Combining Propositions \ref{straight-existence}, \ref{uniqueness}, and \ref{Poiseuille} together finishes the proof of Theorem \ref{bounded channel}.

\section{Flows in channels with unbounded outlets}\label{secexist2}
In this section, we study the flows in channels with unbounded width. Recall the definition of $\beta$ which is given in \eqref{assumpf}. In the rest of this section, $(4\beta)^{-1}$ is used frequently to here and there.  For convenience, denote
\begin{equation*}
	\beta^*:=(4\beta)^{-1}.
\end{equation*}
Clearly, one has
\begin{equation*}
	\|f'\|_{L^\infty}= 2\beta=(2\beta^*)^{-1}
\end{equation*}
and
\begin{equation}\label{4-17}
	\frac12 f(t)\leq f(\xi) \leq \frac32f(t) \text{ for any }\xi\in[t-\beta^* f(t),\,  t+\beta^* f(t)].
\end{equation}
Define
\begin{equation*}
	k(t):=\int_0^t f^{-\frac53}(\xi)\,d\xi
\end{equation*}
and let $h(t)$ be the inverse function of $k(t)$. Then one has
\begin{equation*}
	t=\int_0^{h(t)}f^{-\frac53}(\xi)\,d\xi
\quad \text{and}\quad
	h'(t)=f^{\frac53}(h(t)).
\end{equation*}
Denote
\begin{equation}\label{defh}
	h_L(t)=h(-t)+\beta^*f(h(-t))  \text{ and }h_R(t)=h(t)-\beta^*f(h(t)).
\end{equation}
Direct computations give
\begin{equation}\label{4-1}
	\frac{d}{dt}h_L(t)=-h'(-t)-\beta^*f'(h(-t))h'(-t)=-[1+\beta^*f'(h(-t))]f^\frac53(h(-t))\leq -\frac{\underline{d}^\frac53}{2}
\end{equation}
and
\begin{equation}\label{4-2}
	\frac{d}{dt}h_R(t)=h'(t)-\beta^*f'(h(t))h'(t)=[1-\beta^*f'(h(t))]f^\frac53(h(-t))\ge \frac{\underline{d}^\frac53}{2}.
\end{equation}


The existence of the solutions for problem \eqref{NS}, \eqref{BC} and \eqref{flux constraint} is investigated in three cases, according to  the range of $k$.

{\bf Case 1. The range of $k(t)$ is $(-\infty, \infty)$.} In this case, the function $h(t)$ is defined on $(-\infty,\infty)$. It follows from \eqref{4-1} and \eqref{4-2} that for suitably large $t$, one has
\begin{equation*}
	h_L(t)<h_R(t).
\end{equation*}
 Then we introduce a new truncating function  $\hat{\zeta}( x,t)$ on $\Omega$ as follows,
\begin{equation} \label{cut-off-hat}
\hat{\zeta}( x,t)=\left\{
\begin{aligned}
&0,~~~~~~~~~~~~~~ &&\text{ if }x_1\in (-\infty,h(-t))\cup(h(t),\infty),\\
&\frac{h(t)-x_1}{f(h(t))},~~~~~~&&\text{ if }x_1\in [h_R(t),\, h(t)],\\
&\beta^*,~~~~~~~~~~~~~~ &&\text{ if }x_1\in (h_L(t),\, h_R(t)),\\
&\frac{-h(-t)+x_1}{f(h(-t))},~~~~~~&&\text{ if }x_1\in [h(-t),\, h_L(t)].
\end{aligned}\right.
\end{equation}
For the sake of convenience, one denotes
\begin{equation}\label{4-6}
\hat{\Omega}_{t}=\{ x\in \Omega:~x_1\in (h(-t),h(t))\}\ \ \ \text{ and } \ \ \ \breve{\Omega}_{t}=\hat{\Omega}_{t}\setminus \overline{\hat{ E}},
\end{equation}
where
$\hat{ E}=\hat{ E}^+\cup \hat{ E}^-$
with
\begin{equation}\label{4-7}
\hat{ E}^-=\{ x\in\Omega:x_1\in (h(-t),h_L(t))\},~\hat{ E}^+=\{ x\in\Omega:x_1\in (h_R(t),h(t))\}.
\end{equation}
Clearly, $\nabla \hat{\zeta}$ and $\partial_t\hat{\zeta}$ vanish outside $\hat{ E}$ and satisfy
\begin{equation}\label{4-3}
|\nabla\hat{\zeta}|=|\partial_{x_1}\hat{\zeta}|= [f(h(\pm t))]^{-1}\,\,\text{in}\,\,\hat{E}^\pm,
\end{equation}
and
\begin{equation}\label{4-4}
\partial_t\hat{\zeta}=\frac{h'(\pm t)}{f(h(\pm t))}\left[1\mp\frac{\pm h(\pm t)\mp x_1}{f(h(\pm t))}f'(h(\pm t))\right]\ge \frac12\frac{h'(\pm t)}{f(h(\pm t))}=\frac12[f(h(\pm t))]^\frac23\,\,\text{in}\,\,\hat{E}^\pm.
\end{equation}

With the help of the new truncating function $\hat{\zeta}( x,t)$, we have the following lemma which is used to prove the uniform local estimate for approximate solutions.

\begin{lemma}\label{lemma4}
Assume that the domain $\Omega$ satisfies \eqref{deffbar}, and
\begin{equation*}
	\int_{-\infty}^0 f^{-\frac53}(\tau) \, d\tau = \infty, \ \ \ \ \ \int_0^{+\infty} f^{-\frac53}(\tau) \, d\tau = \infty.
\end{equation*}
Let $\Bv^T$ be the solution of the approximate problem \eqref{aNS} on $\hat{\Omega}_T$, which is obtained in Proposition \ref{appro-existence} and satisfies the energy estimate \eqref{2-10-1}. Then there exists a positive constant $C_{15}$ independent of $t$ and $T$ such that
\begin{equation}\label{4-9}
\|\nabla \Bv^T\|_{L^2(\breve{\Omega}_{t})}^2\leq C_{15}\left(1+\int_{h(-t)}^{h(t)} f^{-3}(\tau)\,d\tau\right) \,\,\text{for any }t^* \leq t\leq T,
\end{equation}
 where
\begin{equation}\label{deft*}
t^*=\sup\{t>0:h_L(t)\ge  h_R(t)\}.
\end{equation}
\end{lemma}
\begin{proof}
The superscript $T$ will be omitted throughout the proof. The proof is quite similar to that for Lemma \ref{lemma3}. Taking the test function $\Bp=\hat{\zeta} \Bv$ in \eqref{2-11} yields
\begin{equation}\label{4-10}
\begin{aligned}
	\int_{\hat{\Omega}_T}\hat{\zeta} |\nabla\Bv|^2\,dx= &\int_{\hat{\Omega}_T }\hat{\zeta} \Bv\cdot\nabla\Bv\cdot \Bg+ \hat{\zeta}  (-\nabla \Bg:\nabla \Bv+\Bg\cdot\nabla\Bv\cdot \Bg)\,dx +\int_{\hat{E}} pv_1\partial_{x_1}\hat{\zeta} \,dx\\
	&+\int_{\hat{E}} \left[\frac12|\Bv|^2(g_1+v_1) +  (g_1+ v_1)\Bv\cdot \Bg -  \partial_{x_1} (\Bg+\Bv)\cdot \Bv  \right] \partial_{x_1}\hat{\zeta} \,dx.
\end{aligned}
\end{equation}

First, using Lemma \ref{lemma1} and choosing sufficiently small $\varepsilon$, one has 
\begin{equation}\label{4-30}
	\begin{aligned}
		\left|\int_{\hat{\Omega}_T }\hat{\zeta}\Bv\cdot\nabla\Bv\cdot\Bg\,dx\right|\leq&
		\left(\int_{\hat{\Omega}_T}\hat{\zeta}|\Bv|^2|\Bg|^2 \,dx\right)^{\frac{1}{2}}\cdot\left(\int_{\hat{\Omega}_T}\hat{\zeta}|\nabla\Bv|^2\,dx\right)^{\frac{1}{2}}\\
		\leq&\frac{1}{2} \int_{\hat{\Omega}_T}\hat{\zeta}|\nabla\Bv|^2\,dx.
	\end{aligned}
\end{equation}
Then using Young's inequality gives 
\begin{equation}\label{4-15}
	\left|\int_{\hat{\Omega}_T}\hat{\zeta} (-\nabla \Bg:\nabla \Bv+\Bg\cdot\nabla\Bv\cdot \Bg)\,dx\right|\leq \frac{1}{4}\int_{\hat{\Omega}_T}\hat{\zeta}|\nabla\Bv|^2\,dx+C \int_{\hat{\Omega}_t} | \nabla \Bg|^2 + |\Bg|^4 \, dx.
	\end{equation}

Furthermore, by Lemma \ref{lemma1}, one has 
\[\|\Bg\|_{L^\infty(\hat{E}^\pm)}\leq [f(h(\pm t))]^{-1}.\]
This, together with \eqref{4-3} and Lemmas \ref{lemmaA1}-\ref{lemmaA2}, shows that
\begin{equation}\label{4-11}
	\begin{aligned}
		 &\int_{\hat{ E}^\pm}\frac12 (v_1+g_1)|\Bv|^2\partial_{x_1}\hat{\zeta}\,dx \\
	 \leq&\frac12 [f(h(\pm t))]^{-1} (\|\Bv\|_{L^4(\hat{ E}^\pm)}^2+\|\Bg\|_{L^\infty(\hat{E}^\pm)}\|\Bv\|_{L^2(\hat{E}^\pm)})\|\Bv\|_{L^2(\hat{ E}^\pm)}\\
		\leq&C[f(h(\pm t))]^{-2}M_1^2(\hat{E}^\pm) \|\nabla \Bv\|_{L^2(\hat{ E}^\pm)}^2+C[f(h(\pm t))]^{-1}M_1(\hat{E}^\pm)M_4^2(\hat{E}^\pm)\|\nabla \Bv\|_{L^2(\hat{ E}^\pm)}^3,
	\end{aligned}
	\end{equation}
	\begin{equation}
		\begin{aligned}
			\int_{\hat{E}^\pm} (g_1+ v_1)\Bv\cdot \Bg  \partial_{x_1}\hat{\zeta}\,dx \leq &[f(h(\pm t))]^{-1}( \|\Bv\|_{L^2(E^\pm)}\|\Bg\|_{L^4(\hat{E}^\pm)}^2+ \|\Bg\|_{L^\infty(\hat{E}^\pm)}\|\Bv\|_{L^2(\hat{E}^\pm)}^2)\\
			\leq&  C[f(h(\pm t))]^{-2}M_1^2(\hat{E}^\pm)\|\nabla\Bv\|_{L^2(\hat{E}^\pm)}^2+ C\|\Bg\|_{L^4(\hat{E}^\pm)}^4,
		\end{aligned}
	\end{equation}
	and 
	\begin{equation}\label{4-12}
		\begin{aligned}
			&\int_{\hat{E}^\pm}  -\partial_{x_1} (\Bg+\Bv)\cdot \Bv  \partial_{x_1}\hat{\zeta}\,dx\\
			 \leq& C[f(h(\pm t))]^{-1}(\|\Bv\|_{L^2(\hat{E}^\pm)}\|\nabla \Bv\|_{L^2(\hat{E}^\pm)}+ C\|\nabla\Bg\|_{L^2(\hat{E}^\pm)}\|\Bv\|_{L^2(\hat{E}^\pm)})\\
			\leq& C([f(h(\pm t))]^{-2}M_1^2(\hat{E}^\pm)+[f(h(\pm t))]^{-1}M_1(\hat{E}^\pm))\|\nabla\Bv\|_{L^2(\hat{E}^\pm)}^2 + C\|\nabla\Bg\|_{L^2(\hat{E}^\pm)}^2.
		\end{aligned}
	\end{equation}
	Finally, one applies Lemmas \ref{lemmaA1}-\ref{lemmaA5} and integration by parts to conclude
\begin{equation*}
	\begin{aligned}
		&\left|\int_{\hat{ E}^\pm}pv_1\partial_{x_1}\hat{\zeta}\,dx\right|=[f(h(\pm t))]^{-1}\left|\int_{\hat{ E}^\pm}p{\rm div}\,\Ba\,dx\right|\\
		=&[f(h(\pm t))]^{-1}\left|\int_{\hat{ E}^\pm}\nabla \Bv:\nabla\Ba+(\Bv\cdot \nabla \Bg +(\Bg+\Bv)\cdot\nabla \Bv-\Delta \Bg+\Bg\cdot \nabla \Bg)\cdot \Ba\,dx\right|\\
		=&[f(h(\pm t))]^{-1}\left|\int_{\hat{ E}^\pm}\nabla \Bv:\nabla\Ba-\Bv\cdot \nabla \Ba\cdot \Bg -(\Bg+\Bv)\cdot\nabla \Ba\cdot \Bv+\nabla \Bg:\nabla \Ba-\Bg\cdot \nabla \Ba\cdot\Bg\,dx\right|\\
		\le &C[f(h(\pm t))]^{-1}\|\nabla\Ba\|_{L^2(\hat{ E}^\pm)}\left(\|\nabla \Bv\|_{L^2(\hat{ E}^\pm)}+\|\Bv\|_{L^4(\hat{ E}^\pm)}^2+\|\nabla\Bg\|_{L^2(\hat{ E}^\pm)}+\|\Bg\|_{L^4(\hat{ E}^\pm)}^2\right)\\
		\leq&C[f(h(\pm t))]^{-1}M_5(\hat{E}^\pm)M_1(\hat{E}^\pm)\|\nabla \Bv\|_{L^2(\hat{ E}^\pm)}\left(\|\nabla \Bv\|_{L^2(\hat{ E}^\pm)}+M_4^2(\hat{E}^\pm)\|\nabla\Bv\|_{L^2(\hat{ E}^\pm)}^2\right.\\
		&\left.+\|\nabla\Bg\|_{L^2(\hat{ E}^\pm)}+\|\Bg\|_{L^4(\hat{ E}^\pm)}^2\right),
	\end{aligned}
\end{equation*}
where $\Ba\in H_0^1(\hat{E}^\pm)$ satisfies
\[
\text{div}~\Ba =v_1 \quad \text{in}\,\, \hat{E}^{\pm}
\]
and
\begin{equation}\label{4-5}
\|\nabla \Ba\|_{L^2(\hat{E}^\pm)}\leq M_5(\hat{E}^\pm)\|v_1\|_{L^2(\hat{E}^\pm)}.
\end{equation}
Here the constant $M_5(\hat{E}^\pm)$ in \eqref{4-5} is uniform with respect to $t$ provided $f'$ is bounded. Then it follows from using Young's inequality that
\begin{equation}\label{4-16}
	\begin{aligned}
		\left|\int_{\hat{ E}^\pm}pv_1\partial_{x_1}\hat{\zeta}\,dx\right|\leq&C[f(h(\pm t))]^{-1}M_1(\hat{E}^\pm) \left(\|\nabla \Bv\|_{L^2(\hat{ E}^\pm)}^2+M_4^2(\hat{E}^\pm)\|\nabla\Bv\|_{L^2(\hat{ E}^\pm)}^3\right)\\
		&+C[f(h(\pm t))]^{-2} M_1^2(\hat{E}^\pm) \|\nabla\Bv\|_{L^2(\hat{ E}^\pm)}^2 + C \int_{\hat{ E}^\pm} |\nabla \Bg|^2 + |\Bg|^4 \, dx .
	\end{aligned}
\end{equation}
Moreover, it follows from Lemmas \ref{lemmaA1} and \ref{lemmaA2} that there exists a uniform constant $C>0$ such that the constants $M_1(\hat{E}^\pm)$ and $M_4(\hat{E}^\pm)$ appeared in \eqref{4-11}-\eqref{4-16} satisfy
\begin{equation*}
C^{-1} f(h(\pm t)) \leq M_1(\hat{E}^\pm)\leq C f(h(\pm t))\ \ \ \mbox{and}\ \ \ C^{-1} [f(h(\pm t))]^\frac12 \leq M_4(\hat{E}^\pm)\leq C [f(h(\pm t))]^\frac12.
\end{equation*}
Define
\begin{equation*}
\hat{y}(t)=\int_{\hat{\Omega}_T}\hat{\zeta} |\nabla \Bv|^2\,dx.
\end{equation*}
By virtue of \eqref{4-4}, we have 
\begin{equation*}\begin{aligned}
	\hat{y}'(t)=&\int_{\hat{\Omega}_T}\partial_t\hat{\zeta} |\nabla \Bv|^2\,dx
	\ge\frac12 [f(h(-t))]^\frac23\int_{\hat{ E}^-} |\nabla \Bv|^2\,dx
	+\frac12 [f(h(t))]^\frac23\int_{\hat{ E}^+} |\nabla \Bv|^2\,dx.
\end{aligned}
\end{equation*}
Using Lemma \ref{lemma1}, one can combine \eqref{4-10}-\eqref{4-16} to conclude
\begin{equation*}
\begin{aligned}
	\hat{y}(t)
	\leq& C\|\nabla \Bv\|_{L^2(\hat{E})}^2+Cf(h(-t))\|\nabla \Bv\|_{L^2(\hat{E}^-)}^3+Cf(h(t))\|\nabla \Bv\|_{L^2(\hat{E}^+)}^3+C\int_{\hat{\Omega}_t } |\nabla \Bg|^2 + |\Bg|^4 \, dx\\
	\leq& C\left\{[f(h(-t))^{-\frac23}+f(h(t))^{-\frac23}]\hat{y}'(t)+  [\hat{y}'(t)]^\frac32\right\}+C\int_{h(-t)}^{h(t)} f^{-3}(\tau)\,d\tau\\
	\leq &C_{11}\left\{\hat{y}'(t)+  [\hat{y}'(t)]^\frac32 \right\}+C_{12}\int_{h(-t)}^{h(t)} f^{-3}(\tau)\,d\tau.
\end{aligned}
\end{equation*}

Define
\begin{equation*}
\hat\Psi(\tau)=C_{11}\left(\tau+ \tau^\frac{3}{2}  \right)\quad
\text{and}
\quad
\hat\varphi(t)=C_{13}+C_{14}\int_{h(-t)}^{h(t)} f^{-3}(\tau)\,d\tau,
\end{equation*}
where $C_{13}$ and $C_{14}$ are large enough such that
\begin{equation*}
C_{12}\int_{h(-t)}^{h(t)} f^{-3}(\tau)\,d\tau\leq \frac12\varphi(t)\ \  \text{ and }\ \
\hat\varphi(t)\ge 2 \hat\Psi(\hat\varphi'(t))\quad \text{for any }t\geq t^*.
\end{equation*}
 This holds since
\begin{equation*}
\begin{aligned}
	|\hat\varphi'(t)|=&C_{14}\left|\frac{d}{dt}\int_{h(-t)}^{h(t)}f^{-3}(\tau)\,d\tau\right|
	=C_{14} \left|\frac{h'( t)}{[f(h( t))]^3} + \frac{h'(-t)}{[f(h(-t))]^3}\right| \\
	\leq& C_{14}[f(h(t))]^{-\frac43} +C_{14}[f(h(-t))]^{-\frac43} \\
	\leq& 2 C_{14} \underline{d}^{-\frac43},
\end{aligned}
\end{equation*}
where $d$ is defined in \eqref{assumpf}. The estimate \eqref{2-10-1} shows
\begin{equation*}
\hat{y}(T)=\|\hat{\zeta}(\cdot, T)^\frac12\nabla \Bv\|_{L^2(\Omega)}^2\leq C_0\int_{\hat{\Omega}_T} | \nabla \Bg|^2 + |\Bg|^4 \, dx \leq \hat\varphi(T),
\end{equation*}
provided  $C_{13}$ and $C_{14}$ are large enough. Hence it follows from Lemma \ref{lemmaA4}  that for any $ t^*\leq t\leq T$, one has
\begin{equation*}
\hat{y}(t)=\|\hat{\zeta}(\cdot, t)^\frac12\nabla \Bv \|_{L^2(\Omega )}^2\leq C_{13}+C_{14}\int_{h(-t)}^{h(t)} f^{-3}(\tau)\,d\tau.
\end{equation*}
 In particular, one has
\begin{equation*}
\|\nabla \Bv \|_{L^2(\breve{\Omega}_t )}^2\leq C_{13}(\beta^*)^{-1}+C_{14}(\beta^*)^{-1}\int_{h(-t)}^{h(t)} f^{-3}(\tau)\,d\tau.
\end{equation*}
This finishes the proof of the lemma.
\end{proof}

With the help of Lemma \ref{lemma4}, one could find at least one  solution of \eqref{NS1} in a way similar to Proposition \ref{straight-existence}.
\begin{pro}\label{unbounded exits-1}  Assume that the domain $\Omega$ satisfies the assumptions in Lemma \ref{lemma4},  the problem \eqref{NS}-\eqref{BC} and \eqref{flux constraint} has a solution $\Bu=\Bv+\Bg\in H_\sigma(\Omega)$ satisfying
\begin{equation}
\|\nabla \Bu\|_{L^2(\breve{\Omega}_{t})}^2\leq C_{16}\left(1+\int_{h(-t)}^{h(t)} f^{-3}(\tau)\,d\tau\right),
\end{equation}
where the constant $C_{16}$ depends only on  $\Phi$ and $\Omega$.
\end{pro}

Next, we prove that the solution $\Bu$ satisfies the estimate \eqref{1-9}.
\begin{pro}\label{unbounded exits-2}
Let $\Bu=\Bv+\Bg$ be the solution obtained in Proposition \ref{unbounded exits-1}. There exists a constant $C_{21}$ depending only on $\Phi$ and $\Omega$ such that for any $t\ge 0$, one has
\begin{equation}\label{est111.5}
\|\nabla \Bu\|_{L^2(\Omega_{0,t})}^2\leq C_{21}\left(1+\int_0^t f^{-3}(\tau )\,d\tau\right)
\end{equation}
and
\begin{equation}\label{est11.6}
\|\nabla \Bu\|_{L^2(\Omega_{-t,0})}^2\leq C_{21}\left(1+\int_{-t}^0 f^{-3}(\tau )\,d\tau\right).
\end{equation}
\end{pro}
\begin{proof}
It's sufficient to prove \eqref{est111.5} since the proof for \eqref{est11.6} is similar.  First, for $t$ suitably large,  we introduce the following truncating function
\begin{equation*}
\hat{\zeta}^+( x,t)=\left\{
\begin{aligned}
	&0, ~~~~~~&&\text{ if }x_1\in (-\infty,\,0),\\
	&\beta^* x_1, ~~~~~~&&\text{ if }x_1\in [0,\,1],\\
	&\beta^*,~~~~~~~~~~~~~~ &&\text{ if }x_1\in (1,\,h_R(t)),\\
	&\frac{h(t)-x_1}{f(h(t))},~~~~~~&&\text{ if }x_1\in [h_R(t),\, h(t)],\\
	&0,~~~~~~~~~~~~~~ &&\text{ if }x_1\in (h(t),\, \infty),
\end{aligned}\right.
\end{equation*}
where $h_R(t)$ is defined in \eqref{defh}. Taking the test function $\Bp=\hat{\zeta}^+ \Bv$ in \eqref{2-11} and following the  proof of  Lemma \ref{lemma4}, one has
\begin{equation*}
\begin{aligned}
	\hat{y}^+\leq& C\left\{\int_{ E_0}|\nabla \Bv|^2\,dx+\left(\int_{ E_0}|\nabla \Bv|^2\,dx\right)^\frac32+(\hat{y}^+)'+[(\hat{y}^+)']^\frac32\right\}+C\int_{0}^{h(t)}f^{-3}(\tau)\,d\tau\\
	\leq& C\left\{1+[(\hat{y}^+)']^\frac32\right\}+C\int_{0}^{h(t)}f^{-3}(\tau)\,d\tau\\
	\leq& C_{17}[ (\hat{y}^+)']^\frac32 +C_{18}\left(1+\int_{0}^{h(t)}f^{-3}(\tau)\,d\tau\right),
\end{aligned}
\end{equation*}
where $E_0 = \{ x\in \Omega:\ 0<x_1 < 1 \}$ and
\begin{equation*}
\hat{y}^+(t)=\int_{\Omega}\hat{\zeta}^+|\nabla \Bv |^2\,dx.
\end{equation*}
Set
\begin{equation*}
 \delta_1=\frac12,\quad  \tilde\Psi(\tau)=C_{17}\tau^\frac32,\quad \text{and}\quad \tilde\varphi(t)=C_{19}+C_{20}\int_{0}^{h(t)}f^{-3}(\tau)\,d\tau.
\end{equation*}
Similar to the proof of Lemma \ref{lemma4}, we choose the constants $C_{19}$ and $C_{20}$ to be sufficiently large such that
\begin{equation*}
C_{18}\left(1+\int_{0}^{h(t)}f^{-3}(\tau)\,d\tau\right)\leq \frac{1}{2}\tilde\varphi(t)\ \ \text{ and }\ \ \tilde\varphi(t)\ge 2 \tilde\Psi(\tilde\varphi'(t)).
\end{equation*}
It also follows from the proof of Lemma \ref{lemma4} that one has
\begin{equation*}
\hat{y}^+(t) \leq C_{13} + C_{14}\int_{h(-t)}^{h(t)} f^{-3}(\tau)\,d\tau
\end{equation*}
and
\begin{equation*}
\left|\frac{d}{dt}\int_{h(-t)}^{h(t)} f^{-3}(\tau)\,d\tau\right|\leq  C  \underline{d}^{-\frac43}.
\end{equation*}
Hence, it holds that
\begin{equation*}
\liminf_{t\rightarrow + \infty} \frac{\hat{y}^+(t) }{\tilde{z}(t)} = 0,
\end{equation*}
where $\tilde{z}(t)=\frac{1}{108C_{17}^2}t^3$ is a nonnegative solution to the ordinary differential equation
\begin{equation*}
\tilde{z}(t)=\delta_1^{-1}\Psi(\tilde{z}'(t))=2C_{17}[\tilde{z}'(t)]^\frac32.
\end{equation*}
It follows from Lemma \ref{lemmaA4} that one has
\begin{equation}\label{4-24}
\hat{y}^+(t)\leq C_{19}+C_{20}\int_{0}^{h(t)} f^{-3}(\tau)\,d\tau.
\end{equation}
With the help of \eqref{4-17}, one has further
\begin{equation}\label{4-25}
\begin{aligned}
	\int_{0}^{h(t)} f^{-3}(\tau)\,d\tau=&\int_{0}^{h_R(t)} f^{-3}(\tau)\,d\tau+\int_{h_R(t)}^{h(t)} f^{-3}(\tau)\,d\tau\\
	\leq&\int_{0}^{h_R(t)} f^{-3}(\tau)\,d\tau+  \max_{\xi\in [h_R(t),\, h(t)]}f^{-3}(\xi )\cdot \beta^* f(h(t))\\
	\leq&\int_{0}^{h_R(t)} f^{-3}(\tau)\,d\tau+ 2^3 \beta^* f^{-2}(h(t))\\
	\leq&\int_{0}^{h_R(t)} f^{-3}(\tau)\,d\tau+  2^3 \beta^*  d ^{-2},
\end{aligned}
\end{equation}
where $h_R(t)$ is defined in \eqref{defh}. Combining  \eqref{4-24} and \eqref{4-25} yields
\begin{equation*}
\|\nabla \Bv\|_{L^2(\Omega_{0,h_R(t)})}^2\leq C_{21}\left(1+\int_{0}^{h_R(t)} f^{-3}(\tau)\,d\tau\right).
\end{equation*}
This, together with Lemma \ref{lemma1}, finishes the proof of the proposition.
\end{proof}

Hence we finish the proof for Part (i) of Theorem \ref{unbounded channel} in the case that the range of $k(t)$ is $(-\infty, \infty)$.


{\bf Case 2. The range of $k(t)$ is $(-L, \, R)$, $0< L, R< \infty.$}  In this case, it holds that
\[
\int_{-\infty}^{+\infty} f^{-\frac53}(\tau) \, d\tau= R+ L < \infty.
\]

Let $v^T$ be the solution of the approximate problem \eqref{aNS} on $\Omega_T$, which is obtained in Proposition \ref{appro-existence} and satisfies \eqref{2-10-1}. Hence, one has
\[
\int_{-\infty}^{+\infty} f^{-\frac53}(\tau) \, d\tau= R+ L < \infty.
\]

With the help of this uniform estimate and Lemma \ref{lemma1}, there exists at least one solution of \eqref{NS1}, which satisfies the estimate
\begin{equation}\label{estimate-case2}
\| \nabla \Bu\|_{L^2(\Omega)}^2  \leq C.
\end{equation}
Hence we finish the proof for Part (i) of Theorem \ref{unbounded channel} in the case that the range of $k(t)$ is $(-L, R)$.

\vspace{3mm}
{\bf Case 3. The range of $k(t)$ is $(-L, \, \infty)$ or $(-\infty, \, R)$, $0<L, R< \infty$.}  Without loss of generality, we assume that
the range of $k(t)$ is $(-L, \infty)$. In this case, $h(t)$ is defined on $(-L,\infty)$. By \eqref{4-2}, one has $h_R(t)=h(t)-\beta^*f(h(t))>0$ for suitably large $t$. Then we introduce the new truncating function as follows,
\begin{equation} \label{cut-off-hat-new}
\hat{\zeta}^L_T ( x,t)=\left\{
\begin{aligned}
		&\beta^*,~~~~~~~~~~~~~~ &&\text{ if }x_1\in (-T,\, h_R(t)),\\
	&\frac{h(t)-x_1}{f(h(t))},~~~~~~&& \text{ if } x_1\in [h_R(t),\, h(t)],\\
	&0,~~~~~~~~~~~~~~ &&\text{ if } x_1\in (h(t),\,\infty),
\end{aligned}\right.
\end{equation}
where $h_R(t)$ is defined in \eqref{defh}.  Denote
\begin{equation}\label{defhatt}
	\hat{t}=\sup\{t>0:\ h_R(t) \leq 0 \}.
\end{equation}

\begin{lemma}\label{lemma-case3}  Assume that the domain $\Omega$ satisfies \eqref{assumpf''},  and
\begin{equation*}
\int_{-\infty}^0 f^{-\frac53}(\tau) \, d\tau = L< \infty, \ \ \ \ \ \int_0^{+\infty} f^{-\frac53}(\tau) \, d\tau = \infty.
\end{equation*}
Let $\Bv^T$ be the solution of the approximate problem \eqref{aNS} in $\Omega_{-T, h(T)}$, which is obtained in Proposition \ref{appro-existence} and satisfies the energy estimate \eqref{2-10-1}. Then there exists a positive constant $C_{22}$ independent of $t$ and $T$ such that for any $\hat{t} \leq t\leq T$, one has
\begin{equation}\label{case3-1}
	\|\nabla \Bv^T\|_{L^2(\Omega_{-T,\,  h_R(t)} )}^2\leq C_{22}\left(1+\int_0^{h(t)} f^{-3}(\tau)\,d\tau\right),
\end{equation}
 where $h_R(t)$ is defined in \eqref{defh} and $C_{22}$ is independent of $T$.
\end{lemma}

\begin{proof}
The superscript $T$ is omitted throughout the proof.
We follow the proof of Lemma \ref{lemma4} by taking the test function $\Bp=\hat{\zeta}^L_T \Bv$ in \eqref{2-11}. Similarly, one has
\begin{equation}\label{case3-4}
\begin{aligned}
	\hat{y}^L(t) & \leq C \left\{ (\hat{y}^L)^{\prime} + \left[ (\hat{y}^L)^\prime    \right]^{\frac32}   \right\} + C \int_{-T}^{h(t)}
	f^{-3}(\tau) \, d\tau \\
	& \leq C \left\{ (\hat{y}^L)^{\prime} + \left[ (\hat{y}^L)^\prime    \right]^{\frac32}   \right\} + C \left(1+ \int_{0}^{h(t)}
	f^{-3}(\tau) \, d\tau \right) ,
\end{aligned}
\end{equation}
where
\begin{equation*}
\hat{y}^L (t) = \int_{\Omega_{-T, h(T)}} \hat{\zeta}^L_T |\nabla \Bv|^2 \, dx .
\end{equation*}
Hence, the same argument as in the proof of Lemma \ref{lemma4} yields
\be \label{case3-6}
\hat{y}^L(t) \leq C \left( 1  + \int_0^{h(t)} f^{-3} (\tau) \, d\tau \right).
\ee
This completes the proof of the lemma.
\end{proof}

\begin{pro}\label{prop-case3}
Assume that the domain $\Omega$ satisfies the assumptions of Lemma \ref{lemma-case3},  the problem \eqref{NS}-\eqref{BC} and \eqref{flux constraint} has a solution $\Bu=\Bv+\Bg\in H_\sigma(\Omega)$ satisfying
\begin{equation}
	\|\nabla \Bu\|_{L^2(\Omega_{0, t} )}^2\leq C_{23}\left(1+\int_0^t f^{-3}(\tau)\,d\tau\right)
\end{equation}
and
\be
\|\nabla \Bu\|_{L^2(\Omega_{-t , 0} )}^2 \leq C_{23},
\ee
where the constant $C_{23}$ depends only on  $\Phi$ and $\Omega$.
\end{pro}
\begin{proof}
With the help of Lemma \ref{lemma-case3}, one can find at least one solution $\Bu = \Bv + \Bg$ of \eqref{NS1} in a way similar to Proposition \ref{straight-existence}. Following the same argument in the estimate \eqref{4-25} yields
\be \label{case3-10}
\begin{aligned}
	 \|\nabla \Bv\|_{L^2 (\Omega_{0, h_R(t)}) }^2  
	\leq &
	C \left(1 +  \int_0^{h(t)} f^{-3}(\tau) \, d\tau \right)\\
	\leq & C \left(1 +  \int_0^{h_R(t)} f^{-3}(\tau) \, d\tau \right) + C \int_{h_R(t)}^{ h(t)} f^{-3}(\tau) \, d\tau \\
	\leq & C_{23} \left(1+  \int_0^{h_R(t)} f^{-3}(\tau) \, d\tau \right),
\end{aligned}
\ee
where $h_R(t)$ is defined in \eqref{defh}. Hence one has
\be \label{case3-11}
\|\nabla \Bv\|_{L^2(\Omega_{0,h_R(t)} )}^2  \leq C_{23} \left(1 +
\int_0^{h_R(t)} f^{-3}(\tau) \, d\tau \right).
\ee
On the other hand, according to Lemma \ref{lemma-case3}, it holds that
\be \label{case3-12}
\|\nabla \Bv\|_{L^2(\Omega_{-T, 0} )}^2 \leq C\left(1 + \int_0^{h({\hat{t}}) } f^{-3}(\tau) \, d\tau\right) \leq C_{23}.
\ee
Combining the estimates \eqref{case3-11}-\eqref{case3-12} and Lemma \ref{lemma1} finishes the proof of the proposition.
\end{proof}

Hence we finish the proof for Part (i) of Theorem \ref{unbounded channel} in the case that the range of $k(t)$ is $(-L, +\infty)$. The same proof applies to the case that the range of $k(t)$ is $(-\infty, R)$. The proof of existence for flows in channels with unbounded outlets is completed.

We are ready to prove the uniqueness of solutions when the flux $\Phi$ is small.
In fact, one can derive some refined estimate for the local Dirichlet norm of $\Bu$, which plays an important role in proving the uniqueness when $\Phi$ is small.


\begin{pro}\label{decay rate-right}
Let $\Bu=\Bv+\Bg$ be the solution obtained in Part (i) of Theorem \ref{unbounded channel}. Assume further that either
\begin{equation}\label{4-27-2}
	\left|\int_0^{\infty} f^{-3}(\tau)\,d\tau\right|=\infty,\ \ \ \lim_{t\to \infty}f'(t)= 0,
\end{equation}
or
\begin{equation}\label{4-27-3}
	\left|\int_0^{\infty}f^{-3}(\tau)\,d\tau\right|<\infty,\ \ \ \lim_{t\to \infty}\frac{\sup_{\tau \ge t} f'(\tau)}{\left|\int_t^{ \infty} f^{-3}(\tau)\,d\tau\right|^\frac12}= 0.
\end{equation}
 Then there exists a constant $C_{31}$ depending only on $ \Phi$, and $\Omega$ such that for any $t\geq 0$, one has
\begin{equation*}
\|\nabla \Bu\|_{L^2(\Omega_{t-\beta^* f(t),t})}^2\leq \frac{C_{31}}{  f^2(t)}.
\end{equation*}

\end{pro}
\begin{proof} We divide the proof into three steps.

{\em Step 1. Truncating function.}
Clearly,
\begin{equation*}
	\frac{d}{dt}(t\pm \beta^* f(t))=1\pm \beta^* f'(t) \ge \frac12.
\end{equation*}
Hence the function $t\pm \beta^* f(t)$ are strictly monotone increasing functions on $\mathbb{R}$.
 For any fixed $T>0$, one can uniquely define the numbers $\hat{T},T_1$, and $T_2$ by
\begin{equation*}
\hat{T}= T - \beta^* f(T), \ \  T_1= \hat{T} - \beta^* f(\hat{T}), \ \ \mbox{and}\ \ T= T_2 - \beta^* f(T_2).
\end{equation*}

 Let $T_0\ge 1$ be a positive constant to be determined.
We introduce two monotone increasing functions $m_i(t)(i=1,2)$ such that for any  $t\in [0,t_1]$,
\begin{equation}\label{4-26}
\left\{\begin{aligned}
&	\frac{d}{dt}m_1(t)=f^{\frac53}(T_1-m_1(t)),\\
&	\frac{d}{dt}m_2(t)=f^{\frac53}(T_2+m_2(t)),\\
&	m_i(0)=0, i=1, 2,
\end{aligned}\right.
\end{equation}
where $t_1$ is the number satisfying
\begin{equation}\label{4-21}
m_1(t_1)=T_1-T_0.
\end{equation}
Noting that  $\frac{d}{dt}m_1(t)\ge \underline{d}^\frac{5}{3}>0$, the number $t_1$ is well-defined. Then  we define the new truncating function $\tilde{\zeta}^+$ as follows,
\begin{equation*}
\tilde{\zeta}^+( x,t)=\left\{
\begin{aligned}
	&\frac{x_1-T_1+m_1(t)}{f(T_1-m_1(t))},~~&&\text{ if }x_1\in [T_1-m_1(t),T_1^*(t)],\\
	&\beta^*,~~~~~~~~~ &&\text{ if }x_1\in (T_1^*(t),
		T_2^*(t)),\\
	&\frac{T_2+m_2(t)-x_1}{f(T_2+m_2(t))},~~&&\text{ if }x_1\in [T_2^*(t),T_2+m_2(t)],\\
	&0,~~~~~~~~ &&\text{ if }x_1\in (-\infty,T_1-m_1(t) )\cup (T_2+m_2(t),\infty),
\end{aligned}\right.
\end{equation*}
where
\[
T_1^*(t)=T_1-m_1(t)+\beta^* f(T_1-m_1(t))\quad \text{and}\quad T_2^*(t)=T_2+m_2(t)-\beta^* f(T_2+m_2(t)).
\]
With the help of \eqref{4-26},  similar to \eqref{4-3}-\eqref{4-4}, one has
\begin{equation*}
|\nabla \tilde{\zeta}^+|=|\partial_{x_1}\tilde{\zeta}^+|=\frac{1}{f(T_i\pm m_i(t))}
\quad\text{and}\quad
|\partial_t\tilde{\zeta}^+|\geq \frac12 [f(T_i\pm m_i(t))]^\frac23  \,\,\text{in}\,\, \operatorname{supp} \nabla \tilde{\zeta}^+=\operatorname{supp} \partial_t \tilde{\zeta}^+.
\end{equation*}

{\em Step 2. Energy estimate. } Taking the test function $\Bp=\tilde{\zeta}^+ \Bv$ in \eqref{2-11} and following the proof of Lemma \ref{lemma4} yield that for any $t\in [0,t_1]$,
\begin{equation}\label{eq130.5}
	\begin{aligned}
\tilde{y}^+\leq & C_{24}\left\{[f^{-\frac23}(T_2+m_2(t))+f^{-\frac23}(T_1-m_1(t))](\tilde{y}^+)'+[(\tilde{y}^+)']^\frac32\right\}\\
 &+C_{25}\int_{T_1-m_1(t)}^{T_2+m_2(t)} f^{-3}(\tau) \,d\tau,
\end{aligned}
\end{equation}
 where
\begin{equation*}
\tilde{y}^+(t)=\int_{\Omega}\tilde{\zeta}^+|\nabla \Bv |^2\,dx.
\end{equation*}
By virtue of Propositions \ref{unbounded exits-2} and \ref{prop-case3}, one has
\begin{equation}\label{4-22-1}
	\begin{aligned}
		\tilde{y}^+(t_1)\leq &C \left(1+\int_{0}^{T_2+m_2(t_1)}f^{-3}(\tau) \,d\tau\right)\\
		\leq& C\int_{T_0}^{T_2+m_2(t_1)}f^{-3}(\tau) \,d\tau+C \left(1+\int_{0}^{T_0}f^{-3}(\tau) \,d\tau\right).
	\end{aligned}
\end{equation}
{\em Step 3. Analysis for flows in channels satisfying \eqref{4-27-2}.}
Firstly, under the assumption \eqref{4-27-2}, choose $T_0$ and $T$ to be sufficiently large such that
\begin{equation}\label{4-23-2}
	1+\int_{0}^{T_0}f^{-3}(\tau)\,d\tau\leq 2\int_{0}^{T_0}f^{-3}(\tau)\,d\tau\leq 2\int_{T_0}^{T_1}f^{-3}(\tau)\,d\tau.
\end{equation}
Recalling that $T_1-m_1(t_1)=T_0$, one uses \eqref{4-22-1} and \eqref{4-23-2} to obtain
\begin{equation*}
	\tilde{y}^+(t_1)\leq C_{26}\int_{T_1-m_1(t_1)}^{T_2+m_2(t_1)}f^{-3}(\tau) \,d\tau.
\end{equation*}
Now, we set $\delta_1=\frac12$,
\begin{equation*}
	\Psi(t,\tau)=C_{24}\left\{ [f^{-\frac23}(T_2+m_2(t))+f^{-\frac23}(T_1-m_1(t))]\tau +\tau ^\frac32 \right\},
\end{equation*}
and
\begin{equation*}
	\varphi(t)=(2C_{25}+C_{26})\int_{T_1-m_1(t)}^{T_2+m_2(t)}f^{-3}(\tau)\,d\tau+C_{27}f^{-2}(T),
\end{equation*}
where $C_{27}$ is to be chosen. Thus, one has
\begin{equation}\label{eqyPsiphi}
	\tilde{y}^+\leq \Psi(t,(\tilde{y}^+)')+\frac12 \varphi(t)
\quad
\text{and}
\quad
	\tilde{y}^+(t_1)\leq \varphi(t_1).
\end{equation}
Moreover, according to \eqref{4-26} and the definition of $\varphi(t)$ and $\Psi(t,\tau)$, it holds that
\begin{equation*}
	\begin{aligned}
		\varphi'(t)
		=&(2C_{25}+C_{26})\left(\frac{m_2'(t)}{f^3(T_2+m_2(t))}+\frac{m_1'(t)}{f^3(T_1-m_1(t))}\right)\\
		=&(2C_{25}+C_{26})\left(f^{-\frac43}(T_2+m_2(t))+f^{-\frac43}(T_1-m_1(t))\right).
	\end{aligned}
\end{equation*}
Therefore, it holds that
\begin{equation*}
	\begin{aligned}
		\Psi(t, \varphi'(t))=&C_{24}\left\{ [f^{-\frac23}(T_2+m_2(t))+f^{-\frac23}(T_1-m_1(t))]\varphi'(t)+[\varphi'(t) ]^\frac32 \right\}\\	
		\leq &C\left[f^{-2}(T_2+m_2(t))+f^{-2}(T_1-m_1(t))\right]\\
		=&C\left[ 2f^{-2}(T)+2\int_{T_1-m_1(t)}^T(f'f^{-3})(\tau)\,d\tau-2\int_T^{T_2+m_2(t)}(f'f^{-3})(\tau)\,d\tau\right]\\
		\leq &2C\left( f^{-2}(T)+\int_{T_1-m_1(t)}^{T_2+m_2(t)}(f'f^{-3})(\tau)\,d\tau\right)\\
		\leq&2C\left( f^{-2}(T)+ \gamma_0(T_0) \int_{T_1-m_1(t)}^{T_2+m_2(t)}f^{-3}(\tau)\,d\tau\right),
	\end{aligned}
\end{equation*}
where
\begin{equation*}
	\gamma_0(T_0): =\sup_{t\ge T_0}|f'(t)|.
\end{equation*}
According to the assumption \eqref{4-27-2}, one could choose sufficiently large $T_0$ and $C_{27}$ such that
\begin{equation}\label{eq133.5}
	\varphi(t)\ge 2\Psi(t, \varphi'(t)).
\end{equation}
Now, it follows from Lemma \ref{lemmaA4} that one has
\begin{equation}\label{4-29-2}
	\tilde{y}^+(t)\leq \varphi(t) \text{ for any }t\in[0,t_1].
\end{equation}

{\em Step 4. Analysis for flows in channels satisfying \eqref{4-27-3}. }
If instead of \eqref{4-27-2}, the assumption \eqref{4-27-3} holds, we choose $T_0$ and $T$ to be sufficiently large such that
\begin{equation*}
	\int_{T_0}^\infty f^{-3}(\tau)\,d\tau\leq 1\quad \text{and}\quad
	\int_{T_0}^{T_1} f^{-3}(\tau)\,d\tau\ge \frac12 \int_{T_0}^\infty f^{-3}(\tau)\,d\tau.
\end{equation*}
Hence, it holds that
\begin{equation}\label{4-23-3}
	1+\int_{0}^{T_0}f^{-3}(\tau)\,d\tau\leq 1+\int_{0}^{\infty}f^{-3}(\tau)\,d\tau\leq \frac{2C}{\int_{T_0}^\infty f^{-3}(\tau)\,d\tau} \int_{T_0}^{T_1}f^{-3}(\tau)\,d\tau.
\end{equation}
Recalling that  $T_1-m_1(t_1)=T_0$, one combines \eqref{4-22-1} and \eqref{4-23-3} to obtain
\begin{equation*}
	\begin{aligned}
		\tilde{y}^+(t_1)\leq &C \int_{T_0}^{T_2+m_2(t_1)}f^{-3}(\tau) \,d\tau+C \left(1+\int_{0}^{T_0}f^{-3}(\tau) \,d\tau\right)\\
		\leq& \frac{C_{26}}{\int_{T_0}^\infty f^{-3}(\tau)\,d\tau}\int_{T_1-m_1(t_1)}^{T_2+m_2(t_1)}f^{-3}(\tau) \,d\tau.
	\end{aligned}
\end{equation*}
Now,  set $\delta_1=\frac12$,
\begin{equation*}
	\Psi(t,\tau)=C_{24}\left([f^{-\frac23}(T_2+m_2(t))+f^{-\frac23}(T_1-m_1(t))]\tau +\tau ^\frac32 \right),
\end{equation*}
and
\begin{equation*}
	\varphi(t)=\left(2C_{25}+\frac{C_{28}}{\int_{T_0}^\infty f^{-3}(\tau)\,d\tau}\right)\int_{T_1-m_1(t)}^{T_2+m_2(t)}f^{-3}(\tau)\,d\tau+C_{29}f^{-2}(T),
\end{equation*}
where $C_{29}$ is to be determined. Then the inequalities in \eqref{eqyPsiphi} still hold.
Moreover, according to \eqref{4-26} and the definition of $\varphi(t)$ and $\Psi(t,\tau)$, one has
\begin{equation*}
	\begin{aligned}
		\varphi'(t)
		=&\left(2C_{25}+\frac{C_{28}}{\int_{T_0}^\infty f^{-3}(\tau)\,d\tau}\right)\left(\frac{m_2'(t)}{f^3(T_2+m_2(t))}+\frac{m_1'(t)}{f^3(T_1-m_1(t))}\right)\\
		=&\left(2C_{25}+\frac{C_{28}}{\int_{T_0}^\infty f^{-3}(\tau)\,d\tau}\right)\left(f^{-\frac43}(T_2+m_2(t))+f^{-\frac43}(T_1-m_1(t))\right).
	\end{aligned}
\end{equation*}
Hence,
\begin{equation*}
	\begin{aligned}
	\Psi(t, \varphi'(t))=&C_{24}\left([f^{-\frac23}(T_2+m_2(t))+f^{-\frac23}(T_1-m_1(t))]\varphi'(t)+[\varphi'(t) ]^\frac32 \right)\\	
		\leq &\frac{C}{\left(\int_{T_0}^\infty f^{-3}(\tau)\,d\tau\right)^\frac32}\left[f^{-2}(T_2+m_2(t))+f^{-2}(T_1-m_1(t))\right]\\	
		=&\frac{2C}{\left(\int_{T_0}^\infty f^{-3}(\tau)\,d\tau\right)^\frac32}\left( f^{-2}(T)+\int_{T_1-m_1(t)}^T(f'f^{-3})(\tau)\,d\tau-\int_T^{T_2+m_2(t)}(f'f^{-3})(\tau)\,d\tau\right)\\
		\leq&\frac{2C}{\left(\int_{T_0}^\infty f^{-3}(\tau)\,d\tau\right)^\frac32}\left( f^{-2}(T)+ \int_{T_1-m_1(t)}^{T_2+m_2(t)}| f'f^{-3}|(\tau)\,d\tau\right)\\
		\leq&2C\left( \frac{f^{-2}(T)}{\left(\int_{T_0}^\infty f^{-3}(\tau)\,d\tau\right)^\frac32} +\frac{\gamma_1(T_0)}{\int_{T_0}^\infty f^{-3}(\tau)\,d\tau}\int_{T_1-m_1(t)}^{T_2+m_2(t)}f^{-3}(\tau)\,d\tau\right),
	\end{aligned}
\end{equation*}
where
\begin{equation*}
	\gamma_1(T_0) =\frac{\sup_{t\ge T_0}|f'(t)|}{\left(\int_{T_0}^\infty f^{-3}(\tau)\,d\tau\right)^\frac12}.
\end{equation*}
According to the assumption \eqref{4-27-3}, one could choose sufficiently large $T_0$ and $C_{29}$ such that \eqref{eq133.5} holds. One can also get \eqref{4-29-2} with the aid of Lemma \ref{lemmaA4}.

{\em Step 5. Growth estimate.}
In particular, taking $t=0$ in \eqref{eq130.5} gives
\begin{equation}\label{4-28}
\|\nabla \Bv\|_{L^2(\Omega_{\hat{T},T})}^2\leq C\int_{T_1}^{T_2}f^{-3}(\tau)\,d\tau+Cf^{-2}(T).
\end{equation}
Finally, using the inequality \eqref{4-17}, one has
\begin{equation}\label{4-29}
\begin{aligned}
	\int_{T_1}^{T_2}f^{-3}(\tau)\,d\tau= &\int_{T_1}^{\hat{T}}f^{-3}(\tau)\,d\tau+\int_{\hat{T}}^Tf^{-3}(\tau)\,d\tau+  \int_{T}^{T_2}f^{-3}(\tau)\,d\tau\\
	\leq& 27\beta^* f^{-3}(\hat{T})f(T_1)+8\beta^* f^{-3}(T)f(T)+ 27 \beta^*f^{-3}(T)f(T_2)\\
	\leq& 54 \beta^* f^{-2}(\hat{T})+8\beta^* f^{-2}(T)+ 27 \beta^* f^{-3}(T) f(T_2)\\
	\leq &Cf^{-2}(T).
\end{aligned}
\end{equation}
Combining \eqref{4-28} and \eqref{4-29} gives
\begin{equation}\label{4-30}
\|\nabla \Bv\|_{L^2(\Omega_{\hat{T},T})}^2\leq \frac{C_{30}}{f^2(T)}.
\end{equation}
This, together with Lemma \ref{lemma1}, finishes the proof of the proposition.
\end{proof}

Similarly, one can also prove the estimate for $t<0$.
\begin{pro}\label{decay rate-left}
Let $\Bu=\Bv+\Bg$ be the solution obtained in Part (i) of Theorem \ref{unbounded channel}. Assume further that either
\begin{equation}\label{4-27-4}
	\int_{-\infty}^{0} f^{-3}(\tau)\,d\tau= \infty,\ \ \ \lim_{t\to -\infty}f'(t)= 0,
\end{equation}
or
\begin{equation}\label{4-27-5}
\int_{-\infty}^{0} f^{-3}(\tau)\,d\tau <\infty,\ \ \ \lim_{t\to -\infty}\frac{\sup_{ \tau \leq t} |f'(\tau) | }{\left|\int_{- \infty}^t  f^{-3}(\tau)\,d\tau\right|^\frac12}= 0.
\end{equation}
 Then there exists a constant $C_{31}$ depending only on $\Phi$, and $\Omega$ such that for any $t\geq 0$, one has
\begin{equation*}
\|\nabla \Bu\|_{L^2(\Omega_{-t,-t+\beta^* f(-t)})}^2\leq \frac{C_{31}}{  f^2(-t)}.
\end{equation*}

\end{pro}

With the help of the decay rate of the local Dirichlet norm of solutions obtained in Propositions \ref{decay rate-right}-\ref{decay rate-left}, we are ready to prove the uniqueness of solution when the flux $\Phi$ is sufficiently small.
\begin{pro}\label{unique-unbounded} Under the assumptions of Propositions \ref{decay rate-right}-\ref{decay rate-left}, there exists a constant $\Phi_2$ such that for any  $\Phi\in [0, \Phi_2)$, the solution $\Bu$ obtained in Part (i) of Theorem \ref{unbounded channel} is unique.
\end{pro}
\begin{proof}We divide the proof into three steps.

{\em Step 1. Set up.}
Assume that $\tilde{\Bu}$ is also a solution of problem \eqref{NS}-\eqref{BC} and \eqref{flux constraint} satisfying
\begin{equation*}
\|\nabla \tilde\Bu \|_{L^2(\Omega_{t})}^2\leq C\left(1+\int_{-t}^t f^{-3}(\tau)\,d\tau\right)\quad \text{for any}\,\, t>0.
\end{equation*}
 Then $\oBu :=
\tilde\Bu-\Bu$ is a weak solution to the problem 
\begin{equation}\label{3-15}
\left\{\begin{aligned}
&-\Delta \oBu +\oBu \cdot \nabla\Bu +\Bu\cdot \nabla \oBu +\oBu\cdot \nabla \oBu +\nabla p=0  ~~~~&\text{ in }\Omega,\\
&{\rm div}~\oBu =0&\text{ in }\Omega,\\
&\oBu =0  &\text{ on } \partial\Omega,\\
&\int_{\Sigma(x_1)}\oBu  \cdot \Bn \,ds=0 &\text{ for any }x_1\in \R.
\end{aligned}\right.
\end{equation}
Let $\hat{\zeta}( x,t)$ be the truncating function defined in \eqref{cut-off-hat}.  Testing the problem \eqref{3-15} by $\hat{\zeta}\oBu $ and using integration by parts $\Omega$ yield
\begin{equation}\label{4-31}
\begin{aligned}
	\int_{\Omega }\hat{\zeta} |\nabla\oBu |^2\,dx=&\int_{\breve{\Omega}_t} \hat{\zeta} \oBu\cdot \nabla\oBu\cdot\Bu\,dx+\int_{\hat{E}} \hat{\zeta} \oBu\cdot \nabla\oBu\cdot\Bu\,dx+ \int_{\Omega}p\overline{u}_1\partial_{x_1}\zeta\,dx\\
&+\int_{\hat{E}}\left[- \partial_{x_1}\oBu\cdot \oBu
+\frac{1}{2}(u_1+\overline{u}_1) |\oBu|^2+(\Bu\cdot \oBu) \overline{u}_1\right]\partial_{x_1}\hat{\zeta}\,dx
\end{aligned}
\end{equation}
where $\breve{\Omega}_t$ and $\hat{E}^\pm$ are defined in \eqref{4-6} and \eqref{4-7}.

{\em Step 2. Estimate for the Dirichlet norm.} Let $\breve{\Omega}_t^i=\{ x\in\Omega:~x_1\in(A_{i-1},A_i),~i=1,2,\cdots, N(t)\}$. Here the sequence $\{A_i\}$ satisfies $h_L(t)=A_0<\cdots<A_j=0 <A_{j+1}<\cdots < A_{N(t)}=h_R(t)$,
\begin{equation*}
\frac{\beta^*}{2}f(A_i) \leq  A_{i+1}-A_{i}\leq \beta^* f(A_{i}) \text{ for any }0\leq i\leq j-1
\end{equation*}
and
\begin{equation*}
\frac{\beta^*}{2} f(A_{i+1}) \leq  A_{i+1}-A_{i}\leq \beta^* f(A_{i+1})\text{ for any }j\leq i\leq  N(t)-1.
\end{equation*}
By Lemmas \ref{lemmaA1}, \ref{lemmaA2}, and \ref{lemma1}, and Propositions \ref{decay rate-right}-\ref{decay rate-left}, one has
\begin{equation}\label{4-32}
\begin{aligned}
	\int_{\breve{\Omega}_{t}}\hat{\zeta}\oBu \cdot\nabla \oBu \cdot \Bu\,dx
	\leq & \beta^* \sum_{i=1}^{N(t)}\int_{\breve{\Omega}_t^i }|\oBu \cdot\nabla \oBu \cdot \Bu|\,dx \\
	\leq &  \beta^* \sum_{i=1}^{N(t)}\|\nabla\oBu \|_{L^2(\breve{\Omega}_t^i )} \|\oBu \|_{L^4(\breve{\Omega}_t^i )}(\|\Bv\|_{L^4(\breve{\Omega}_t^i )}+\|\Bg\|_{L^4(\breve{\Omega}_t^i )}) \\
	\leq&\beta^* \sum_{i=1}^{N(t)}\|\nabla\oBu \|_{L^2(\breve{\Omega}_t^i )}^2(M_4^2\|\nabla\Bv\|_{L^2(\breve{\Omega}_t^i)}+M_4\|\Bg\|_{L^4(\breve{\Omega}_t^i )}) \\
	\leq&C\sum_{i=1}^{j}\|\nabla\oBu \|_{L^2(\breve{\Omega}_t^i )}^2(f(A_{i-1})\cdot f^{-1}(A_{i-1})+f^{\frac12}(A_{i-1})f^{-\frac12}(A_{i-1})) \\
	&+C\sum_{i=j+1}^{N(t)}\|\nabla\oBu \|_{L^2(\breve{\Omega}_t^i )}^2(f(A_i)\cdot f^{-1}(A_i)+f^{\frac12}(A_i)f^{-\frac12}(A_i)) \\
	\leq&C_{32}\sum_{i=1}^{N(t)}\|\nabla\oBu \|_{L^2(\breve{\Omega}_t^i )}^2 \\
	=&C_{32}\int_{\breve{\Omega}_{t}} |\nabla\oBu |^2 \, dx,
\end{aligned}
\end{equation}
where the constant $C_{32}$ goes to zero as $\Phi\to 0$. Hence there exists a $\Phi_2>0$, such that for any $\Phi\in [0,\Phi_2)$, one has
\begin{equation}\label{4-33}
\int_{\breve{\Omega}_{t}}\hat{\zeta}\oBu \cdot\nabla \oBu \cdot \Bu\,dx \leq \frac{1}{2}\int_{\hat{\Omega}_t}\hat{\zeta}|\nabla\oBu |^2\,dx.
\end{equation}

On the other hand, using Lemmas \ref{lemmaA1} and \ref{lemmaA2} yields
\begin{equation}\label{4-34}
\begin{aligned}
	&\int_{\hat{ E}^\pm}\left[-\hat{\zeta}\partial_{x_1}\oBu\cdot \oBu+ (\oBu \cdot \Bu)\overline{u}_1+\frac12|\oBu |^2(u_1+\overline{u}_1) \right] \partial_{x_1}\hat{\zeta}\,dx 
	 + \int_{\hat{ E}^\pm} \hat{\zeta}  \oBu\cdot\nabla \oBu\cdot\Bu \, dx \\
	\leq&C[f(h(\pm t))]^{-1}\left[ \|\nabla\oBu \|_{L^2(\hat{ E}^\pm)}\|\oBu \|_{L^2(\hat{ E}^\pm)}+\|\oBu \|_{L^4(\hat{ E}^\pm)}^2 \left(\|\Bu\|_{L^2(\hat{ E}^\pm)}+\|\oBu\|_{L^2(\hat{ E}^\pm)} \right)\right]\\
	&+\beta^* \|\oBu\|_{L^4(\hat{ E}^\pm)}\|\nabla \oBu\|_{L^2(\hat{ E}^\pm)}\|\Bu\|_{L^4(\hat{ E}^\pm)}\\
	\leq&C[f(h(\pm t))]^{-1}\left(M_1(\hat{ E}^\pm)\|\nabla\oBu \|_{L^2(\hat{ E}^\pm)}^2+M_4^2(\hat{ E}^\pm)\|\nabla \oBu \|_{L^2(\hat{ E}^\pm)}^2\|\Bu\|_{L^2(\hat{ E}^\pm)}\right)\\
	&+C[f(h(\pm t))]^{-1}M_1(\hat{ E}^\pm)M_4^2(\hat{ E}^\pm)\|\nabla\oBu\|_{L^2(\hat{ E}^\pm)}^3+\beta^* M_4(\hat{ E}^\pm)\|\nabla \oBu\|_{L^2(\hat{ E}^\pm)}^2\|\Bu\|_{L^4(\hat{ E}^\pm)}.
\end{aligned}
\end{equation}

Similar to \eqref{4-16}, one can estimate the term
 $\int_{\hat{ E}^\pm}p\overline{u}_1\partial_{x_1}\zeta\,dx$. More precisely,
\begin{equation}\label{4-35}
\begin{aligned}
	&\left|\int_{\hat{ E}^\pm}p\overline{u}_1\partial_{x_1}\hat{\zeta}\,dx\right|=\left|\int_{\hat{ E}^\pm}\partial_{x_1}\hat{\zeta} p{\rm div}\,\Ba\,dx\right|=[f(h(\pm t))]^{-1}\left|\int_{\hat{ E}^\pm} p{\rm div}\,\Ba\,dx\right|\\
	=&[f(h(\pm t))]^{-1}\left|\int_{\hat{ E}^\pm} \nabla\overline {\Bu}:\nabla\Ba+(\oBu \cdot \nabla \Bu +(\Bu+\oBu)\cdot \nabla \oBu )\cdot\Ba\,dx\right|\\
	=&[f(h(\pm t))]^{-1}\left|\int_{\hat{ E}^\pm} \nabla\overline {\Bu}:\nabla\Ba-\oBu \cdot \nabla \Ba\cdot\Bu -(\Bu+\oBu)\cdot \nabla \Ba \cdot \oBu\,dx\right|\\
	\le &C[f(h(\pm t))]^{-1}\|\nabla\Ba\|_{L^2(\hat{ E}^\pm)}\left(\|\nabla \overline {\Bu}\|_{L^2(\hat{ E}^\pm)}+\|\overline {\Bu}\|_{L^4(\hat{ E}^\pm)}\|\Bu\|_{L^4(\hat{ E}^\pm)}+\|\overline {\Bu}\|_{L^4(\hat{ E}^\pm)}^2\right)\\
	\le &C[f(h(\pm t))]^{-1}M_1(\hat{ E}^\pm)M_5(\hat{ E}^\pm)
\left(\|\nabla \overline {\Bu}\|_{L^2(\hat{ E}^\pm)}^2+M_4(\hat{ E}^\pm)\|\nabla\overline {\Bu}\|_{L^2(\hat{ E}^\pm)}^2\|\Bu\|_{L^4(\hat{ E}^\pm)}\right.\\
&\left.+M_4^2(\hat{ E}^\pm)\|\nabla \oBu\|_{L^2(\hat{ E}^\pm)}^3\right),
\end{aligned}
\end{equation}
where $\Ba\in H_0^1(\hat{E}^\pm)$  satisfies
\[
\text{div}~\Ba= \overline{u}_1\quad \text{in}\,\,\hat{E}^{\pm}
\]
and
\begin{equation*}
\|\nabla\Ba\|_{L^2(\hat{E}^\pm)} \leq M_5(\hat{ E}^\pm)\|\overline{u}_1\|_{L^2(\hat{E}^\pm)}.
\end{equation*}

Note that for the subdomain $\hat{ E}^\pm$,  $M_5(\hat{E}^\pm)$ is a uniform constant and  the constants $M_1(\hat{E}^\pm),M_4(\hat{E}^\pm)$ appeared in \eqref{4-34}-\eqref{4-35} satisfy the following estimates,
\begin{equation}\label{4-36}
 C^{-1} f(h(\pm t)) \leq M_1(\hat{ E}^\pm)\leq C f(h(\pm t)),  \ \  C^{-1} [f(h(\pm t))]^\frac12\leq M_4(\hat{ E}^\pm)\leq C [f(h(\pm t))]^\frac12.
\end{equation}
Moreover, according to Lemmas \ref{lemmaA1}-\ref{lemmaA2}, and \ref{lemma1}, and Propositions \ref{decay rate-right}-\ref{decay rate-left}, one has
\begin{equation}\label{4-37}
\begin{aligned}
	\|\Bu\|_{L^4(\hat{ E}^\pm)}\leq &\|\Bv\|_{L^4(\hat{ E}^\pm)}+\|\Bg\|_{L^4(\hat{ E}^\pm)}
	\leq M_4(\hat{ E}^\pm)\|\nabla\Bv\|_{L^2(\hat{ E}^\pm)}+\|\Bg\|_{L^4(\hat{ E}^\pm)}\\
	\leq& C[f(h(\pm t))]^{-\frac12}
\end{aligned}
\end{equation}
and
\begin{equation}\label{4-38}
\|\Bu\|_{L^2(\hat{ E}^\pm)}\leq \|\Bv\|_{L^2(\hat{ E}^\pm)}+\|\Bg\|_{L^2(\hat{ E}^\pm)}\leq M_1(\hat{ E}^\pm)\|\nabla\Bv\|_{L^2(\hat{ E}^\pm)}+\|\Bg\|_{L^2(\hat{ E}^\pm)}\leq C.
\end{equation}

{\em Step 3. Growth estimate.} Let
\begin{equation*}
\hat{y}(t)=\int_{\Omega}\hat{\zeta }|\nabla \oBu |^2\,dx.
\end{equation*}
Combining \eqref{4-31}-\eqref{4-38} gives the differential inequality
\begin{equation*}
\hat{y}\leq C\left[\hat{y}'+(\hat{y}')^\frac32\right].
\end{equation*}
It follows from Lemma \ref{lemmaA4} that one has either $\oBu=0$ or
\begin{equation*}
\liminf_{t\rightarrow + \infty} t^{-3}\hat{y}(t)>0.
\end{equation*}
Hence the proof of the proposition is completed.
\end{proof}
Combining Propositions \ref{unbounded exits-1}, \ref{unbounded exits-2}, and \ref{unique-unbounded} together finishes the proof of Theorem \ref{unbounded channel}.

\section{The flow converges at the point at infinity}
In this section, we study the pointwise decay rate of the velocity $\Bu$ obtained in Theorem \ref{unbounded channel}. Following the proof  of \cite[Theorem  \uppercase\expandafter{\romannumeral13}.1.1]{Ga}, one could also show that both the solution $\Bu$ obtained in Theorem \ref{unbounded channel} and the corresponding pressure $p$ are smooth in $\overline{\Omega}$.  Furthermore, as is proved in Propositions \ref{decay rate-right},  the Dirichlet norm of $\Bu$ satisfies 
\begin{equation*}
\|\nabla \Bu\|_{L^2(\Omega_{t-\beta^* f(t),t})}\leq \frac{C}{  f(t)}~~~\text{ for any }t\ge 0.
\end{equation*}
Then pointwise decay of $\Bu$ follows from a precise estimate of the high-order norm and the Sobolev embedding theorem.

First, we introduce the following lemma on the interior regularity of solutions to the Stokes equations, whose proof can be found in \cite{Ga}.
\begin{lemma}\label{interior regularity}
	Assume that $\Omega$ is an arbitrary domain in $\R^n$ with $n\ge 2$. Let $\Bu$ be weakly divergence-free with $\nabla\Bv\in L^q_{loc}{(\Omega)}$,$1<q<\infty$,and satisfying 
	\[\int_{\Omega}\nabla\Bv:\nabla\boldsymbol{\varphi}\,dx=\int_{\Omega} \Bf\cdot  \boldsymbol{\varphi}  \,dx  ~~\text{ for any }\boldsymbol{\varphi}\in C_{0,\sigma}^{\infty}{(\Omega)}.\]
	
	If $\Bf\in W_{loc}^{m,q}{(\Omega)}$ for some $ m\ge 0$, then it follows that $\Bv \in W_{loc}^{m+2,q}{(\Omega)}, p \in W_{loc}^{m+1,q}{(\Omega)}$, where $p$ is the pressure associated to $\Bv$. Further the following inequality holds:
\begin{equation}\label{6-3}
	\|\nabla^{m+2}\Bv\|_{L^{q}{(\Omega')}}+\|\nabla^{m+1}p\|_{L^{q}{(\Omega')}}\le C\left(\|\Bf\|_{W^{m,q}{(\Omega'')}}+\|\Bv\|_{W^{1,q}{(\Omega''\setminus\Omega')}}+\|p\|_{L^{q}{(\Omega''\setminus\Omega')}}\right),
\end{equation}
where $\Omega'$, $\Omega''$ are arbitrary bounded subdomains of $\Omega$ with $\overline{\Omega'}\subset\Omega''$, $\overline{\Omega''}\subset \Omega$,
and $C=C(n,q,m,\Omega',\Omega'')$.
\end{lemma}

\begin{remark}\label{qe}
If the domain $\Omega''\setminus\Omega'$, in the previous lemma, satisfies the cone condition, we can remove the term involving the pressure on the right-hand side of \eqref{6-3} by modifying $p$ with a constant. Therefore, we obtain 
\begin{equation}\label{6-5}
	\begin{aligned}
	\|\nabla^{m+2}\Bv\|_{L^{q}{(\Omega')}}+ \|\nabla^{m+1}p\|_{L^{q}{(\Omega')}}\le C\left(\|\Bf\|_{W^{m,q}{(\Omega'')}}+\|\Bv\|_{W^{1,q}{(\Omega''\setminus\Omega')}}\right).
	\end{aligned}
\end{equation}
To see this, we denote $D=\Omega''\setminus\Omega'$ for simplicity and let $\phi\in L^{q'}{(D)}$ be arbitrary, where $q'$ is the conjugate index of $q$.  Note that $\phi- \phi_D \in L^{q'}$ satisfies 
\[\int_{D} \phi-  \phi_D \,dx =0,\]
where $\phi_D=\frac{1}{|D|}\int_D \phi\,dx$. Then, according to Lemma \ref{lemmaA5}, the problem 
\[\operatorname{div} \boldsymbol{\Phi}= \phi- \phi_D\]
has at least one solution $\boldsymbol{\Phi}\in W^{1,q'}_0(D)$ such that 
\begin{equation}\label{6-6}
	\|\boldsymbol{\Phi}\|_{W^{1,q'}{(D)}}\le C\|\phi\|_{L^{q'}{(D)}}.
\end{equation}

Furthermore, the pressure $p$ associated to $\Bu$ satisfies
\[\int_D \nabla\Bv:\nabla\boldsymbol{\psi}\,dx= \int_D \Bf\cdot \boldsymbol{\psi}  \,dx +\int_D p \operatorname{div} \boldsymbol{\psi} \,dx~~~~\text{ for any }\boldsymbol{\psi} \in W_0^{1,q'}(D).\]
Taking the test function $\boldsymbol{\psi}=\boldsymbol{\Phi}$ and using integration by parts, we obtain 
\begin{equation*}
	\int_D (p-p_D)\phi\,dx=-\int_D \Bf\cdot \boldsymbol{\Phi}  \,dx+ \int_D \nabla\Bv:\nabla\boldsymbol{\Phi}\,dx,
\end{equation*}
where $p_D= \frac{1}{|D|} \int_D  p\,dx$. This, together with \eqref{6-6}, gives  
\[\begin{aligned}
	\left|\int_D (p-p_D)\phi\,dx\right|\leq &\|\Bf\|_{L^q(D)} \|\boldsymbol{\Phi}\|_{L^{q'}(D)}+\|\nabla \Bu\|_{L^q(D)} \|\nabla \boldsymbol{\Phi}\|_{L^{q'}(D)}   \\
	\leq& C(\|\Bf\|_{L^q(D)} +\| \Bv\|_{W^{1,q}(D)}) \|\phi\|_{L^{q'}(D)}.
\end{aligned}\]
By the arbitrariness of $\phi$, we deduce that
\begin{equation}\label{star}
	\begin{aligned}
		\|p-p_D\|_{L^{q}{(D)}}\le C\left(\|\Bf\|_{L^{q}{(D)}}+\|\Bv\|_{W^{1,q}{(D)}}\right).
	\end{aligned}
\end{equation}
Substituting \eqref{star} into \eqref{6-3}, we obtain \eqref{6-5}.
\end{remark}
With the aid of Lemma \ref{interior regularity}, one could improve the interior regularity of the solutions to the Navier-Stokes equation \eqref{NS}, by considering the nonlinear term $\Bu\cdot\nabla\Bu$ as the external force term $\Bf$. For any $\delta\in(0,\frac12)$, define 
\[\Omega_{\delta f}=\{x\in\Omega:~x_2\in (f_1(x_1)+\delta f(x_1),f_2(x_1)-\delta f(x_1))\},\]
where $f=f_2-f_1$ . Then we obtain the decay rate of $\Bu(x)$ for $x$ away from the boundary.
\begin{pro}\label{interior decay}
Let $\Bu=\Bv+\Bg$ be the solution obtained in Part (i) of Theorem \ref{unbounded channel}. Assume further that either \eqref{1-16} or \eqref{1-17} holds. Then for any $\delta\in (0,\frac12 )$  and $x\in \Omega_{\delta f}$, one has 
\[|\Bu(x)|\leq \frac{C\delta^{-1}}{f(x_1)},\]
where $C$ is a constant depending only on  $\Omega$, and $\Phi$.
\end{pro}
\begin{proof}
Fix any  $x^*=(x_1^*,x_2^*)\in \Omega_{\delta f}$ with $x_1^*>0$. Then there exists some $t>0$ such that $x_1^*=\frac{1}{2}(2t-\beta^* f(t))$. Moreover, one can verify that 
\[B_{r}(x^*) \subset \Omega_{t-\beta^*f(t),t} \text{ for any }r\leq \frac{\delta f(x_1^*)}{1+\beta}=:r_0.\]
Now we set $r=\frac{1}{4}r_0$ and denote 
\begin{equation}\label{scaling}
			\Bu^r(z)=r\Bu(x^*+rz),~~ p^r(z)=r^2p(x^*+rz).
	\end{equation}
The scaling property of Navier-Stokes system implies that  $(\Bu^r, p^r)$ is also a solution to the Navier-Stokes equations in $B_{2}(0)$, that is 
	\begin{equation}\label{NSv}
		\left\{
		\begin{aligned}
			&-\Delta \Bu^r+\Bu^r\cdot \nabla \Bu^r +\nabla p^r=0 ~~~~&\text{ in }B_{2}(0),\\
			&{\rm div}~\Bu^r=0&\text{ in }B_{2}(0).
		\end{aligned}\right.
	\end{equation}
Then we could apply Lemma \ref{interior regularity} with $\Omega''=B_2(0)$ and $\Omega'=B_1(0)$. In particular, taking $q=\frac{4}{3}, m=0$ in \eqref{6-5} and using Sobolev embedding inequality, we have 
\begin{equation}\label{6-8}
 \begin{aligned}
		\|\Bu^r\|_{L^\infty(B_1(0))}\leq &C	\|\Bu^r\|_{W^{2,\frac{4}{3}}{(B_1(0))}} \le C\left(\|\Bu^r\cdot\nabla\Bu^r\|_{L^{\frac{4}{3}}{(B_2(0))}}+\|\Bu^r\|_{W^{1,\frac{4}{3}}(B_2(0))}\right)\\
		\le& C\left(\|\nabla\Bu^r\|_{L^2{(B_2(0))}}  \|\Bu^r\|_{L^4{(B_2(0))}}+\|\Bu^r\|_{W^{1,\frac{4}{3}}(B_2(0))}\right)\\
		\le& C\left(\|\nabla\Bu^r\|_{L^2{(B_2(0))}}  \|\Bu^r\|_{W^{1,2}{(B_2(0))}}+\|\Bu^r\|_{W^{1,2}(B_2(0))}\right).
	\end{aligned}
\end{equation}
Straightforward computations give 
\[	\|\Bu^r\|_{L^\infty(B_1(0))}=r\|\Bu \|_{L^\infty(B^r(x^*))},~~\|\Bu^r\|_{L^2(B_2(0))}= \|\Bu\|_{L^2(B_{2r}(x^*))},\]
and 
\[\|\nabla \Bu^r\|_{L^2(B_2(0))} = r \|\nabla\Bu\|_{L^2(B_{2r}(x^*))}.\]
Then one has 
\begin{equation*}
\|\Bu \|_{L^\infty(B^r(x^*))} \leq Cr^{-1}(r\|\nabla\Bu\|_{L^2(B^r(x^*))}+1)(\| \Bu\|_{L^2(B^r(x^*))}+r\|\nabla\Bu\|_{L^2(B^r(x^*))}).
\end{equation*}
Finally, using Propositions \ref{decay rate-right} and Lemma \ref{lemmaA1}, it follows that 
\begin{equation*}
\begin{aligned}
\|\Bu \|_{L^\infty(B_{r}(x^*))} 
\leq&Cr\|\nabla\Bu\|_{L^2(B_{r}(x^*))}^2+  C\|\nabla\Bu\|_{L^2(B_{r}(x^*))}\\
&+C\|\nabla\Bu\|_{L^2(B_{r}(x^*))}\| \Bu\|_{L^2(B_{r}(x^*))}+ Cr^{-1}\| \Bu\|_{L^2(B_{r}(x^*))}\\
\leq& \frac{C\delta^{-1}}{f(x_1^*)},
\end{aligned}
\end{equation*}
since $\frac12f(t)\leq f(x_1^*) \leq \frac{3}{2} f(t)$.  The case that $x_1^*<0$ is similar. Then we finish the proof of this proposition.
\end{proof}

Next, we will consider the decay rate of $\Bu(x)$ for $x$ near the boundary $\partial\Omega$, that is, $x\in \Omega\setminus \Omega_{\delta f}$.  To this end, we introduce the following lemma on the regularity of solution to the Stokes equations in the half space $\R^2_+$.  See \cite{Ga} for the proof.

\begin{lemma}\label{bpm}
Assume that $m\ge 0$ and $1<q<\infty$. For every 
	\[
	\Bf\in W^{m,q}{(\R_+^n)}\text{ and } g \in W^{m+1,q}(\R^n_+),
	\]
 there exists a pair of functions $(\Bv, p)$ such that
\[
\Bv \in W^{m+2,q}{(Q)}, \quad p\in W^{m+1,q}{(Q)},
\]
for all open cubes $Q \subset \R_+^n$, solving a.e. the following nonhomogeneous Stokes system
\begin{equation}
	 \left\{\begin{aligned}
	-\Delta \Bv+\nabla p=&\Bf&\text{ in } \R^n_+,\\
	\nabla \cdot \Bv =&g&\text{ in } \R^n_+,\\
	\Bv=&0& \text{ on } \partial \R^n_+.
	\end{aligned}\right.
\end{equation} 
Moreover, for all $l\in[0,m]$, we have
\begin{equation}\label{bmpeq}
	\begin{aligned}
		\|\nabla^{l+2}\Bv\|_{L^{q}{(\R^n_+)}}+\|\nabla^{l+1}p\|_{L^{q}{(\R^n_+)}}\le C\left( \|\nabla^{l}\Bf\|_{L^{q}{(\R^n_+)}}+ \|\nabla^{l+1}g\|_{L^{q}{(\R^n_+)}} \right).
	\end{aligned}
\end{equation}
where $C=C(n,q,m)$. 
\end{lemma}

Then the following proposition gives the decay rate of $\Bu(x)$ for $x$ near the boundary.
\begin{pro}\label{boundary decay}
Let $\Bu=\Bv+\Bg$ be the solution obtained in Part (i) of Theorem \ref{unbounded channel}. Assume further that either \eqref{1-16} or \eqref{1-17} holds. Then there exists a constant $\delta_0$  such that for any $\delta\leq \delta_0$ and $x\in \Omega\setminus\Omega_{\delta f}$, one has 
\[|\Bu(x)|\leq \frac{C}{f(x_1)}.\]
Here $C$ is a constant depending only on $\Omega$, and $\Phi$.
\end{pro}
\begin{proof} We divide the proof into several steps.

{\em Step 1.} Fix any  $x^*=(x_1^*,x_2^*)\in \Omega\setminus \Omega_{\delta f}$. Without loss of generality, we assume that  the point $x^*$  with $x_1^*>0$ is near the upper boundary, that is, 
\begin{equation}\label{7-1}
0<f_2(x_1^*)-x_2^*<\delta f(x_1^*).
\end{equation}
There exists some $t$ such that $x_1^*=\frac{1}{2}(2t-\beta^* f(t))$.

Let $\bar{x}^*=(x_1^*,f_2(x_1^*))$ be the corresponding point of $x$ on the upper boundary. Similar to \eqref{scaling}, we introduce the scaling function
\begin{equation} 
			\Bu^r(z)=r\Bu(\bar{x}^*+rz) \text{ and } p^r(z)=r^2p(\bar{x}^*+rz).
	\end{equation}
Then for any $r>0$, $(\Bu^r,p^r)$ satisfies
	\begin{equation}\label{NSv}
		\left\{
		\begin{aligned}
			&-\Delta \Bu^r+\Bu^r\cdot \nabla \Bu^r +\nabla p^r=0  \\
			&{\rm div}~\Bu^r=0
		\end{aligned}\right.~~~~\text{ in }\Omega^r_{t-\beta^*f(t),t},
	\end{equation}
	where
	\[\Omega^r_{t-\beta^*f(t),t}=\{z:~\bar{x}^*+rz \in \Omega_{t-\beta^*f(t),t}\}.\]
	Note that $\Omega^r_{t-\beta^*f(t),t}$ is also a channel type domain of the form 
	\[\Omega^r_{t-\beta^*f(t),t}=\{z:~r^{-1}(t-\beta^*f(t)-x_1^*)<z_1< r^{-1}(t-x_1^*),~f_1^r(z_1)<z_2<f_2^r(z_1)\},\]
	where $f_i^r(z_1)=r^{-1}\left(f_i(rz_1+x_1^*)-f_2(x_1^*)\right)$ for $i=1,2$.

	In the rest of the proof, we choose $r=f(x_1^*)$ so that 
\[\frac12 f(t)\leq r\leq \frac32f(t).\]	
	Here we give some properties of domain $\Omega_{t-\beta^*f(t),t}^r$, which will be used later. Clearly, the points $x^*,\bar{x}^*$ become $z^*:=(0,r^{-1}(x_2^*-f_2(x_1^*)))$ and $(0,0)$ in the $z$-coordinate, respectively. Due to \eqref{7-1}, we have 
	\[|z^*| =r^{-1}(x_2^*-f_2(x_1^*))\leq   2\delta \frac{f(x_1^*)}{f(t)}\leq 3\delta.\]
	According to the assumptions \eqref{assumpf}  and \eqref{assumpf''} on $f_i$, one has 
	\begin{equation}\label{7-3}
	 |(f_i^r)'(z_1)|\leq \|f_i'\|_{L^\infty}=\beta
	\end{equation}
	and 
	\begin{equation}\label{7-4}
	|(f_i^r)''(z_1)|=|rf_i''(rz_1+x_1^*)|\leq 3|(ff_i'')(rz_1+x_1^*)| =3\gamma,
	\end{equation}
	for any $z\in \Omega_{t-\beta^*f(t),t}^r$. Furthermore, the width of the domain $\Omega_{t-\beta^*f(t),t}^r$ satisfies 
	\[\frac{2}{3} \beta^* \leq r^{-1} \beta^*f(t) \leq 2\beta^*.\]
	Then for any $z\in \Omega_{t-\beta^*f(t),t}^r$, we have 
	\begin{equation}
	    |f_2^r(z_1)| \leq \beta \cdot \beta^*=\frac{1}{4}  \text{ and }
	    |f_1^r(z_1)+1| \leq \frac{1}{4},
	\end{equation}
since $f_2^r(0)=0$ and $f_1^r(0)=-1$. Finally, we denote
\[d_0=\min\left\{\frac{1}{3}\beta^*,\frac12\right\},\]
so that the ball $B_{d_0}(0)$ is above the lower boundary of $\Omega_{t-\beta^* f(t)}^r$.
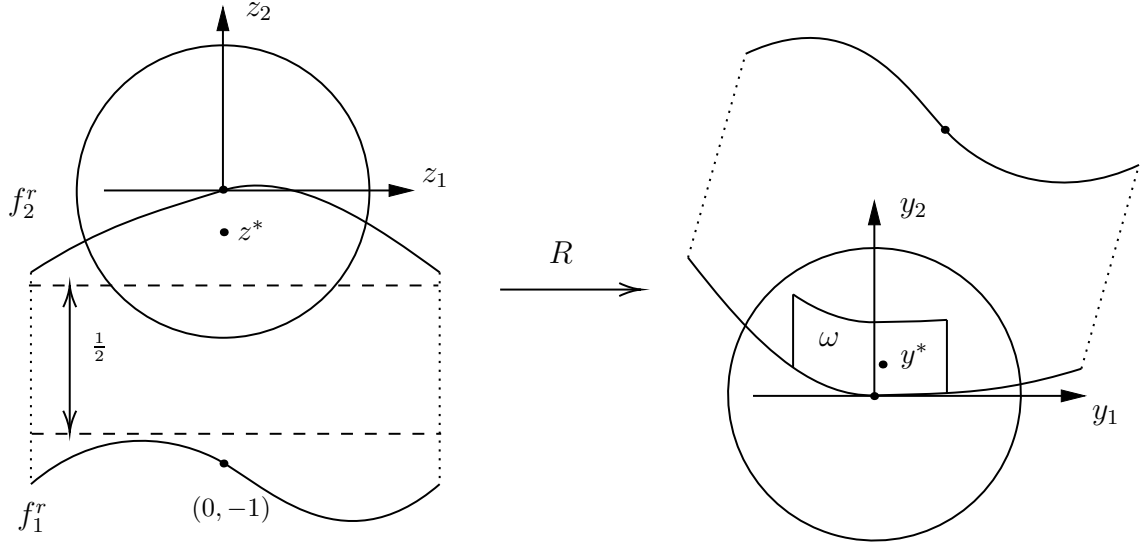
\begin{figure}[h]
	\centering

	\tikzset{every picture/.style={line width=0.75pt}} 

	\begin{tikzpicture}[x=0.75pt,y=0.75pt,yscale=-1,xscale=1]
	
	\draw    (60.36,139.37) .. controls (97,115.56) and (120.37,109.51) .. (157.19,98.04) .. controls (194,86.56) and (236.2,117.31) .. (266.51,139.37) ;
	\draw    (60.36,246.28) .. controls (92.43,218.49) and (127.95,220.98) .. (152.4,232.81) .. controls (176.85,244.65) and (216.12,289.93) .. (266.51,246.28) ;
	\draw  [line width=3] [line join = round][line cap = round] (157.2,97.64) .. controls (157.2,97.64) and (157.2,97.64) .. (157.2,97.64) ;
	\draw  [line width=3] [line join = round][line cap = round] (157.48,235.85) .. controls (157.48,235.85) and (157.48,235.85) .. (157.48,235.85) ;
	\draw  [line width=3] [line join = round][line cap = round] (157.62,119.12) .. controls (157.62,119.12) and (157.62,119.12) .. (157.62,119.12) ;
	\draw  [dash pattern={on 0.84pt off 2.51pt}]  (60.36,139.37) -- (60.36,246.28) ;
	\draw  [dash pattern={on 0.84pt off 2.51pt}]  (266.51,139.37) -- (266.51,246.28) ;
	\draw  [dash pattern={on 4.5pt off 4.5pt}]  (59.35,145.99) -- (265.83,145.99) ;
	\draw  [dash pattern={on 4.5pt off 4.5pt}]  (60.77,220.86) -- (267.25,220.86) ;
	\draw    (79.93,147.99) -- (79.93,218.17) ;
	\draw [shift={(79.93,220.17)}, rotate = 270] [color={rgb, 255:red, 0; green, 0; blue, 0 }  ][line width=0.75]    (10.93,-3.29) .. controls (6.95,-1.4) and (3.31,-0.3) .. (0,0) .. controls (3.31,0.3) and (6.95,1.4) .. (10.93,3.29)   ;
	\draw [shift={(79.93,145.99)}, rotate = 90] [color={rgb, 255:red, 0; green, 0; blue, 0 }  ][line width=0.75]    (10.93,-3.29) .. controls (6.95,-1.4) and (3.31,-0.3) .. (0,0) .. controls (3.31,0.3) and (6.95,1.4) .. (10.93,3.29)   ;
	\draw   (83.46,98.59) .. controls (83.46,57.83) and (116.51,24.78) .. (157.27,24.78) .. controls (198.03,24.78) and (231.08,57.83) .. (231.08,98.59) .. controls (231.08,139.35) and (198.03,172.39) .. (157.27,172.39) .. controls (116.51,172.39) and (83.46,139.35) .. (83.46,98.59) -- cycle ;
	\draw    (590.14,188.01) .. controls (563.04,196.42) and (543.36,199.14) .. (522.33,200.3) .. controls (510.98,200.93) and (499.24,201.1) .. (485.73,201.46) .. controls (471.11,201.85) and (457.4,196.19) .. (444.75,187.4) .. controls (424.03,173.01) and (406.14,150.25) .. (391.76,131.95) ;
	\draw    (619.22,85.14) .. controls (580.8,103.15) and (547.29,91.1) .. (526.98,73.06) .. controls (506.67,55.02) and (481.19,0.77) .. (420.83,29.07) ;
	\draw  [line width=3] [line join = round][line cap = round] (485.61,201.84) .. controls (485.61,201.84) and (485.61,201.84) .. (485.61,201.84) ;
	\draw  [line width=3] [line join = round][line cap = round] (521.25,67.43) .. controls (521.25,67.43) and (521.25,67.43) .. (521.25,67.43) ;
	\draw  [line width=3] [line join = round][line cap = round] (489.91,185.87) .. controls (489.91,185.87) and (489.91,185.87) .. (489.91,185.87) ;
	\draw  [dash pattern={on 0.84pt off 2.51pt}]  (590.14,188.01) -- (619.22,85.14) ;
	\draw  [dash pattern={on 0.84pt off 2.51pt}]  (391.76,131.95) -- (420.83,29.07) ;
	\draw   (556.82,220.98) .. controls (545.73,260.2) and (504.95,283.01) .. (465.72,271.93) .. controls (426.5,260.84) and (403.69,220.06) .. (414.77,180.83) .. controls (425.86,141.61) and (466.64,118.8) .. (505.87,129.88) .. controls (545.09,140.97) and (567.9,181.75) .. (556.82,220.98) -- cycle ;
	\draw    (590,201.72) -- (424.53,201.72) ;
	\draw [shift={(592,201.72)}, rotate = 180] [fill={rgb, 255:red, 0; green, 0; blue, 0 }  ][line width=0.08]  [draw opacity=0] (12,-3) -- (0,0) -- (12,3) -- cycle    ;
	\draw    (485.73,201.46) -- (485.73,105.37) ;
	\draw [shift={(485.73,103.37)}, rotate = 90] [fill={rgb, 255:red, 0; green, 0; blue, 0 }  ][line width=0.08]  [draw opacity=0] (12,-3) -- (0,0) -- (12,3) -- cycle    ;
	\draw    (444.75,187.4) -- (444.75,150.47) ;
	\draw    (522.33,200.3) -- (522.33,163.37) ;
	\draw    (522.33,163.37) .. controls (501.31,164.53) and (499.24,164.17) .. (485.73,164.53) .. controls (472.21,164.89) and (457.4,159.26) .. (444.75,150.47) ;
	\draw    (297,148.06) -- (366,148.06) ;
	\draw [shift={(368,148.06)}, rotate = 180] [color={rgb, 255:red, 0; green, 0; blue, 0 }  ][line width=0.75]    (10.93,-3.29) .. controls (6.95,-1.4) and (3.31,-0.3) .. (0,0) .. controls (3.31,0.3) and (6.95,1.4) .. (10.93,3.29)   ;
	\draw    (97.19,98.04) -- (251,98.04) ;
	\draw [shift={(253,98.04)}, rotate = 180] [fill={rgb, 255:red, 0; green, 0; blue, 0 }  ][line width=0.08]  [draw opacity=0] (12,-3) -- (0,0) -- (12,3) -- cycle    ;
	\draw    (157.19,98.04) -- (157.19,7.37) ;
	\draw [shift={(157.19,5.37)}, rotate = 90] [fill={rgb, 255:red, 0; green, 0; blue, 0 }  ][line width=0.08]  [draw opacity=0] (12,-3) -- (0,0) -- (12,3) -- cycle    ;

	\draw (47.19,95.53) node [anchor=north west][inner sep=0.75pt]    {$f_{2}^{r}$};
	\draw (51.92,253.98) node [anchor=north west][inner sep=0.75pt]    {$f_{1}^{r}$};
	\draw (139.87,250.8) node [anchor=north west][inner sep=0.75pt]  [font=\footnotesize]  {$( 0,-1)$};
	\draw (89.33,168.17) node [anchor=north west][inner sep=0.75pt]  [font=\scriptsize]  {$\frac{1}{2}$};
	\draw (163.12,111.53) node [anchor=north west][inner sep=0.75pt]  [font=\small]  {$z^{*}$};
	\draw (456.48,168.07) node [anchor=north west][inner sep=0.75pt]    {$\omega $};
	\draw (497.5,173.37) node [anchor=north west][inner sep=0.75pt]  [font=\small]  {$y^{*}$};
	\draw (320,122.27) node [anchor=north west][inner sep=0.75pt]    {$R$};
	\draw (256,85.4) node [anchor=north west][inner sep=0.75pt]    {$z_{1}$};
	\draw (167,1.4) node [anchor=north west][inner sep=0.75pt]    {$z_{2}$};
	\draw (497,100.4) node [anchor=north west][inner sep=0.75pt]    {$y_{2}$};
	\draw (594,205.12) node [anchor=north west][inner sep=0.75pt]    {$y_{1}$};

	\end{tikzpicture}
\caption{Scaling and rotation of the domain}	

\end{figure}

{\em Step 2.} Next, we introduce a new coordinate system by rotation such that the unit outer normal vector of $\partial\Omega_{t-\beta^* f(t),t}^r$ at $z=(0,0)$ becomes $(0,-1)$. In fact, we define the transformation $z\mapsto y$ as follows,
\begin{equation}\label{6-11}
	\begin{aligned}
	 \begin{pmatrix}
		y_1\\y_2
	\end{pmatrix}=\begin{pmatrix}
		R_{11} &R_{12}  \\
		R_{21}& R_{22} \\
	\end{pmatrix} \begin{pmatrix}
		z_1 \\z_2 
	\end{pmatrix}	=R \begin{pmatrix}
	z_1 \\z_2
\end{pmatrix}.
	\end{aligned}
\end{equation}
Here $R$ is a rotation matrix satisfying $R\cdot R^T=I$, and $R_{11}=R_{22},R_{12}=-R_{21}$.

In the original $z$-coordinate system, the unit outer normal vector at $z=(0,0)$ is
 \[\Bn^r=\left(\frac{-(f_{2}^r)'(0)}{\sqrt{|(f_{2}^r)'(0)|^2+1}},\frac{1}{\sqrt{|(f_{2}^r)'(0)|^2+1}}\right).\] 
 Noting that 
 \begin{equation}
		 \begin{pmatrix}
			0\\-1
		\end{pmatrix}=\begin{pmatrix}
	R_{11}	& R_{12} \\
	R_{21}	& R_{22} \\
	\end{pmatrix}
	\begin{pmatrix}
		\frac{-(f_{2}^r)'(0)}{\sqrt{|(f_{2}^r)'(0)|^2+1}}\\\frac{1}{\sqrt{|(f_{2}^r)'(0)|^2+1}}
	\end{pmatrix},
\end{equation}
then we can figure out a specific representation of the rotation matrix, that is, 
\begin{equation}
    \begin{pmatrix}
	R_{21}\\R_{22}
\end{pmatrix}=\begin{pmatrix}
\frac{(f_{2}^r)'(0)}{\sqrt{|(f_{2}^r)'(0)|^2+1}}\\-\frac{1}{\sqrt{|(f_{2}^r)'(0)|^2+1}}
\end{pmatrix}.
\end{equation}
Let
\begin{equation}
	\begin{aligned}
		F(y_1,y_2) =f_2^r(z_1)-z_2= f_{2}^r( R_{11}y_1+R_{21}y_2)-(R_{12}y_1+R_{22}y_2).
	\end{aligned}
\end{equation}
Straightforward computations give that $F(0,0)=0$ and 
\begin{equation} \label{7-2}
	\begin{aligned}
\p_ {y_2} F(y_1,y_2)=&R_{21}(f_{2}^r)'(z_1)-R_{22}=\frac{(f_{2}^r)'(0)(f_{2}^r)'(z_1)+1}{\sqrt{|(f_{2}^r)'(0)|^2+1}}\\
=&\frac{|(f_{2}^r)'(0)|^2+1+[(f_{2}^r)'(z_1)-(f_{2}^r)'(0)](f_{2}^r)'(0)}{\sqrt{|(f_{2}^r)'(0)|^2+1}}.
	\end{aligned}
\end{equation}
In particular, we have $\p_ {y_2} F(0,0)=\sqrt{|(f_{2}^r)'(0)|^2+1}>0$. According to the implicit function theorem, there exists a positive constant $L_0>0$ and a $C^1$ function $y_2=\zeta(y_1)$ defined on $[-L_0,L_0]$ such that
\[F(\zeta(y_1),y_1)=0 \text{ for any }y_1\in [-L_0,L_0].\]

{\em Step 3.} We claim that the function $\zeta(y_1)$ can be extended to the interval $[-\frac{1}{4M},-\frac{1}{4M}]$, 
where \[M:=\max\left\{\beta+3\gamma, \frac{\sqrt{2}}{2d_0}\right\}\]
is a constant independent of the choice $x^*$. To prove the claim, we define 
\[
L_1 =\sup\left\{L:~ \zeta(y_1) \text{ is well-defined on } [-L,L]\text{ and }|\zeta(y_1)|<\frac{1}{4M}\text{ for any }y_1\in [-L,L]\right\}.
\] 
It is sufficient to show that $L_1\ge \frac{1}{4M}$. Suppose $L_1<\frac{1}{4M}$. Then $\zeta(y_1)$ is well-defined on $(- L_1, L_1)$ and satisfies $|\zeta(y_1)|<\frac{1}{4M}$ for any $y_1\in (- L_1, L_1)$. In fact, $\zeta(y_1)$ is also well-defined at the endpoints $y_1=\pm L_1 $. Indeed, for any  $y_1\in (- L_1, L_1)$, we have 
\begin{equation}\label{gj}
\begin{aligned}|\zeta'(y_1)|=& \left|\frac{\partial_{y_1}F}{\partial_{y_2}F}\right|=\left|\frac{(f_2^r)'(z_1)-(f_2^r)'(0)}{(f_2^r)'(z_1) (f_2^r)'(0)+1}\right|\\
\leq& \frac{|(f_2^r)'(z_1)-(f_2^r)'(0)|}{1+\frac12 |(f_2^r)'(0)|^2-\frac12 |(f_2^r)'(z_1)-(f_2^r)'(0)|^2 }\\
\leq& \frac{M|z_1|}{1-\frac12 M^2|z_1|^2}\\
\leq&\frac{4 \sqrt{2} }{15},
\end{aligned}\end{equation}
since $|z_1|\leq \sqrt{(y_1)^2+(y_2)^2}\le \frac{\sqrt{2}}{4M}$. It follows from \eqref{gj} that $\displaystyle \lim_{ y_1 \to \pm L_1} \zeta(y_1)$ exists and 
\[
|\zeta(y_1)|\leq \frac{4 \sqrt{2} }{15}\cdot \frac{1}{4M}=\frac{\sqrt{2}}{15M} \text{ for any }y_1\in [- L_1, L_1].
\]
Furthermore, it follows from \eqref{7-2} that one has 
\begin{equation}\label{7-5}
\begin{aligned}
\partial_{ y_2} F( y_1,\zeta(y_1))=&\frac{|(f_{2}^r)'(0)|^2+1+[(f_{2}^r)'(z_1)-(f_{2}^r)'(0)](f_{2}^r)'(0)}{\sqrt{|(f_{2}^r)'(0)|^2+1}}\\
\ge& \frac{1+\frac12 |(f_{2}^r)'(0)|^2-\frac12|(f_{2}^r)'(z_1)-(f_{2}^r)'(0)|^2}{\sqrt{|(f_{2}^r)'(0)|^2+1}}\\
\ge& \frac{1+\frac12 |(f_{2}^r)'(0)|^2-\frac12 M^2|z_1|^2}{\sqrt{|(f_{2}^r)'(0)|^2+1}}\\
\ge&\frac12\sqrt{|(f_2^r)'(0)|^2+1}>0,
\end{aligned}
\end{equation}
for any $y_1\in [-L_1,L_1]$. Using implicit function theorem again, one concludes that there exists some $0<\varepsilon_0<\frac{1}{8M}$ such that $\zeta(y_1)$ is well-defined on $[- L_1-\varepsilon_0, L_1+\varepsilon_0]$ and  $|\zeta(y_1)|\leq \frac{3}{8M}$ for any $y_1\in [- L_1-\varepsilon_0, L_1+\varepsilon_0]$. Similar to \eqref{gj}, for any $y_1\in  [- L_1-\varepsilon_0, L_1+\varepsilon_0]$, we have 
\begin{equation}
\begin{aligned}
|\zeta'(y_1)|
\leq& \frac{M|z_1|}{1 -\frac12 M^2|z_1|^2}\leq \frac{24\sqrt{2}}{55},
\end{aligned}
\end{equation}
since $|z_1| \leq \sqrt{|y_1|^2+|y_2|^2} \leq \frac{3\sqrt{2}}{8M}$. It follows that one has 
\[
|\zeta(y_1)\le |\zeta'(\xi)||L_1+\varepsilon_0|\leq \frac{24\sqrt{2}}{55}\cdot  \frac{ 3}{8M}<\frac{1}{4M}.
\]
This leads to a contradiction. 

According to the argument above, one concludes that
\begin{equation}\label{7-6}
    |\zeta(y_1)|\leq \frac{1}{4M} \text{ and }|\zeta'(y_1)|\leq \frac{4\sqrt{2}}{15} \text{ on  }  [-\frac{1}{4M},\frac{1}{4M}].
\end{equation}
Furthermore, using \eqref{7-3} and \eqref{7-4}, one has 
\begin{equation*}
\left\{\begin{aligned}
&|\partial_{y_1 }F| =\frac{|(f_{2}^r)'(0)-(f_2^r)'(z_1)|}{\sqrt{|(f_{2}^r)'(0)|^2+1}} \leq 3\gamma |z_1|\leq \frac{3\gamma\sqrt{2}}{4M}, \\
&|\partial_{y_1y_1}F| =\frac{|(f_2^r)''(z_1)|}{|(f_{2}^r)'(0)|^2+1} \leq 3\gamma,
\\
&|\partial_{y_2 }F| =\frac{|(f_{2}^r)'(0)(f_2^r)'(z_1)+1|}{\sqrt{|(f_{2}^r)'(0)|^2+1}} \leq 1+\beta^2,\\
&|\partial_{y_2y_2}F| = \frac{|(f_{2}^r)'(0)|^2|(f_2^r)''(z_1)|}{ |(f_{2}^r)'(0)|^2+1 } \leq 3\gamma.
\end{aligned}\right. \text{ in  }[-\frac{1}{4M},\frac{1}{4M}]^2.
\end{equation*}
These, together with \eqref{7-5}, give that 
\begin{equation}\label{7-7}
|\zeta''(y_1)|=\frac{|\partial_{y_1y_1}F |\partial_{y_2}F|^2- 2\partial_{y_1y_2}F\partial_{y_1 }F\partial_{y_2}F+\partial_{y_2y_2}F |\partial_{y_1}F|^2 |}{|\partial_{y_2}F|^3}\leq  C(M,\beta,\gamma) 
\end{equation}
on  $[-\frac{1}{4M},\frac{1}{4M}]$. In particular,  by \eqref{gj}, one has $\zeta'(0)=0$.

{\em Step 4.} Denote $d=\frac{1}{16M}$ and define the truncation domains
\[
\omega=\{y\in \R^2:~|y_1|<d,~\zeta(y_1)<y_2<\zeta(y_1)+d\}
\]
and 
\[
\omega'=\left\{y\in \R^2:~|y_1|<\frac{d}{2},~\zeta(y_1)<y_2<\zeta(y_1)+\frac{d}{2}\right\}.
\]

For any $y\in \omega$, one has 
\[|y|=\sqrt{|y_1|^2+|y_2|^2}\leq \sqrt{d^2+(|\zeta(y_1)|+d)^2}\leq  \frac{\sqrt{26}}{16M}\leq \frac{\sqrt{13} d_0}{8}.\]
Then $\omega \subset B_{d_0}(0)$. Recalling the definition of $F(y_1,y_2)$ and noting
\[	\partial_{y_2}F(y_1,y_2)=\frac{(f_{2}^r)'(0)(f_{2}^r)'(z_1)+1}{\sqrt{|(f_{2}^r)'(0)|^2+1}}>0 \text{ for any }y\in \omega,\]
one has 
\[F(y_1,y_2)=f_2^r(z_1)-z_2>0 \text{ in }\omega,\]
since $F(y_1,\zeta(y_1))=0$. This implies  $\omega\subset R(\Omega^r_{t-\beta^*f(t),t})$, where 
\[R(\Omega^r_{t-\beta^*f(t),t})=\{Rz:~z\in \Omega^r_{t-\beta^*f(t),t}\}.\]
On the other hand, since $|z^*|\leq3\delta$, the corresponding point  $y^*:=y(z^*) $ belongs to $B_{4\delta}(0)\cap R(\Omega^r_{t-\beta^*f(t),t})$. 
Thus, there exists a constant $\delta_0$ depending only on $d$ such that $y^*\in \omega'$ as long as $\delta\leq \delta_0$.

{\em Step 5. }
Define 
\[\Bu_R(y)=R \Bu^r(R^{-1}y) \text{ and } p_R(y)=R  p^r( R^{-1}y).\]
Then $(\Bu_R,p_R)$ satisfies the Navier-Stokes system
\begin{equation*} 
	\left\{
	\begin{aligned}
		&-\Delta  \Bu_R+\Bu_R\cdot \nabla \Bu_R+\nabla p_R=0 ~~~~&\text{ in } \omega ,\\
		&\nabla \cdot \Bu_R =0&\text{ in } \omega ,\\
		&\Bu_R=0&\text{ on }\{(y_1,\zeta(y_1)):~|y_1|<d\}.
	\end{aligned}\right.
\end{equation*}

To flatten the boundary, let us introduce the new variables.
\[
	s_1=y_1,\qquad  s_2=y_2-\zeta(y_1).
\]
Then $\omega, \omega'$ are transformed into the rectangles
\[
\hat{\omega}=\left\{(s_1,s_2)\in \R^2:~|s_1|<d,~0<s_2<d\right\}
\]
and
\[
\hat{\omega}=\left\{(s_1,s_2)\in \R^2:~|s_1|<\frac{d}{2},~0<s_2<\frac{d}{2}\right\},
\]
respectively. Let $\hat{\Bu}(s_1,s_2)=\Bu_R(y_1,y_2)$ and  $\hat{p}(s_1,s_2)=p_R(y_1,y_2)$. Then $(\hat{\Bu},\hat{p})$ satisfies
\begin{equation}\label{NS2}
	\left\{
	\begin{aligned}
		&-\Delta \hat{\Bu}+\nabla  \hat{p}=\hat{\Bf} ~~~~&\text{ in }\hat{\omega},\\
		&\nabla \cdot\hat{\Bu}=\hat{g}&\text{ in }\hat{\omega},\\
		&\hat{\Bu}=0&\text{ on }\{(s_1,0):~|s_1|<d\} 
	\end{aligned}\right.
\end{equation}
where 
\begin{equation*}
	\begin{aligned}
	\hat{\Bf} &=-\hat{\Bu}\cdot\nabla \hat{\Bu}+ \zeta'\hat{u}_1 \partial_{s_2}\hat{\Bu} +\left(\zeta'\partial_{s_2}\hat{p} ,0\right)-2\zeta'\partial_{s_1s_2}\hat{\Bu} +|\zeta'|^2\partial_{s_2s_2} \hat{\Bu} -\zeta'' \partial_{s_2}\hat{\Bu}  
	\end{aligned}
\end{equation*}
and
\[
\hat{g}=\zeta'\partial_{s_2}\hat{u}_1.
\]

Let $\varphi\in C^{\infty}(\R^2_+)$ be a smooth cut-off function such that $\varphi=0 $ in $\R^2_+\setminus\hat{\omega}$ and $\varphi = 1 $ in  $\hat{\omega}'$ with 
\[
\hat{\omega}'=\left\{(s_1,s_2):~|s_1|<\frac{d}{2},\text{ }0<s_2<\frac{d}{2}\right\}.
\]
Moreover, $\varphi$ satisfies
\begin{equation}\label{7-8}
	|\nabla \varphi|\leq \frac{C}{d} \text{  and }|\nabla \varphi|\leq \frac{C}{d^2}.
\end{equation} 

Now we set 
\[
\Bw= \varphi \hat{\Bu},~\pi= \varphi \hat{p}
\]
and extend $\Bw, \pi$ by zero to the half space $\R_+^2$. Straightforward computations show that $(\Bw,\pi)$ satisfies the Stokes system
\begin{equation}\label{NS3}
	\left\{
	\begin{aligned}
		&-\Delta\Bw+\nabla\pi=\Bf ~~~~&\text{ in }\R^2_+,\\
		&\nabla\cdot\Bw=g&\text{ in }\R^2_+,\\
		&\Bv=0&\text{ on }\partial\R^2_+ 
	\end{aligned}\right.
\end{equation}
in a weak sense, where 
\begin{equation}\label{F}
\begin{aligned}
\Bf=&-\hat{\Bu}\Delta \varphi-2\nabla\varphi\cdot \nabla \hat{\Bu} +\hat{p}\nabla \varphi+\varphi \hat{\Bf}\\
=&-\hat{\Bu}\Delta \varphi-2\nabla\varphi\cdot \nabla \hat{\Bu} +\hat{p}\nabla \varphi+\varphi(-\hat{\Bu}\cdot\nabla \hat{\Bu}+ \zeta'\hat{u}_1 \partial_{s_2}\hat{\Bu}  -\zeta'' \partial_{s_2}\hat{\Bu} )\\
&+\varphi(\left(\zeta'\partial_{s_2}\hat{p} ,0\right)-2\zeta'\partial_{s_1s_2}\hat{\Bu} +|\zeta'|^2\partial_{s_2s_2} \hat{\Bu})\\
=&-\hat{\Bu}\Delta \varphi-2\nabla\varphi\cdot \nabla \hat{\Bu} +\hat{p}\nabla \varphi+\varphi(-\hat{\Bu}\cdot\nabla \hat{\Bu}+ \zeta'\hat{u}_1 \partial_{s_2}\hat{\Bu}  -\zeta'' \partial_{s_2}\hat{\Bu} )-\left(\zeta'\hat{p}\partial_{s_2}\varphi  ,0\right)\\
&+2\zeta'(\partial_{s_1}\varphi \partial_{s_2}\hat{\Bu}+ \partial_{s_2}\varphi \partial_{s_1}\hat{\Bu}+ \partial_{s_1s_2}\varphi \hat{\Bu}) -|\zeta'|^2(2\partial_{s_2}\varphi \partial_{s_2}\hat{\Bu}+ \partial_{s_2s_2}\varphi \hat{\Bu})\\
&+ \left(\zeta'\partial_{s_2}\pi ,0\right)-2\zeta'\partial_{s_1s_2}\Bw +|\zeta'|^2\partial_{s_2s_2} \Bw\\
\end{aligned}
\end{equation}
and
\begin{equation}\label{G}
	\begin{aligned}
	g=\hat{\Bu}\cdot \nabla\varphi+\zeta'\partial_{s_2} w_1-\hat{u}_1\zeta'\partial_{s_2}\varphi.
	\end{aligned}
\end{equation}
Using \eqref{7-6}-\eqref{7-8}, straightforward computations give
\begin{equation}\label{FF}
	\begin{aligned}
	\|\Bf\|_{L^{\frac{4}{3}}{(\R^2_+)}}\le& C(1+d^{-2})\left(\| \hat{\Bu}\cdot\nabla \hat{\Bu}\|_{L^{\frac{4}{3}}{(\hat{\omega})}}+\|\hat{\Bu}\|_{W^{1,\frac{4}{3}}{(\hat{\omega})}} +\|\hat{p}\|_{L^{\frac{4}{3}}{(\hat{\omega})}}\right)\\
	&+\|\zeta'\nabla \pi\|_{L^{\frac{4}{3}}{(\R^2_+)}} + \|\left(|\zeta'|^2+2|\zeta'|\right)\nabla^2\Bw\|_{L^{\frac{4}{3}}{(\R^2_+)}}
	\end{aligned}
\end{equation}
and
 \begin{equation}\label{GG}
 	\begin{aligned}
 		\|\nabla g\|_{L^\frac{4}{3} (\R^2_+)}\le C(1+d^{-2})\|\hat{\Bu}\|_{W^{1,\frac{4}{3}}{(\hat{\omega})}}+\|\zeta'\nabla^2\Bw\|_{L^{\frac{4}{3}}{(\R^2_+)}}.
 	\end{aligned}
 \end{equation}

In particular, using \eqref{7-6}, we have 
\[
|\zeta'|^2+2|\zeta'|<1 \text{ and }|\zeta'|<1.
\]
Then it follows from Lemma \ref{bpm} that
\begin{equation}\label{ees}
	\begin{aligned}
		\|\hat{\Bu}\|_{W^{2,\frac{4}{3}}{(\hat{\omega}')}}+\|\hat{p}\|_{W^{1,\frac{4}{3}}{(\hat{\omega}')}}\le C\left(\| \hat{\Bu}\cdot\nabla \hat{\Bu}\|_{L^{\frac{4}{3}}{(\hat{\omega})}}+\|\hat{\Bu}\|_{W^{1,\frac{4}{3}}{(\hat{\omega})}}+\|\hat{p}\|_{L^{\frac{4}{3}}{(\hat{\omega})}}\right).
	\end{aligned}
\end{equation}
Similar to Remark \ref{qe}, one could also remove the term involving the pressure on the right-hand side of \eqref{ees} and obtain
\begin{equation}
	\begin{aligned}
			\|\hat{\Bu}\|_{W^{2,\frac{4}{3}}{(\hat{\omega}')}} \le C\left(\| \hat{\Bu}\cdot\nabla \hat{\Bu}\|_{L^{\frac{4}{3}}{(\hat{\omega})}}+\|\hat{\Bu}\|_{W^{1,\frac{4}{3}}{(\hat{\omega})}}\right).
	\end{aligned}
\end{equation}
Using Sobolev  embedding inequality, we have 
\begin{equation}\label{7-9}
	\begin{aligned}
			\|\hat{\Bu}\|_{L^\infty(\hat{\omega}')} \leq & C\|\hat{\Bu}\|_{W^{2,\frac{4}{3}}{(\hat{\omega}')}} \leq C  \left(\| \hat{\Bu}\cdot\nabla \hat{\Bu}\|_{L^{\frac{4}{3}}{(\hat{\omega})}}+\|\hat{\Bu}\|_{W^{1,\frac{4}{3}}{(\hat{\omega})}}\right)\\
			\leq& C  \left(\| \hat{\Bu} \|_{L^4{(\hat{\omega})}}\| \nabla \hat{\Bu}\|_{L^2 (\hat{\omega}) }+\|\hat{\Bu}\|_{W^{1,2}{(\hat{\omega})}}\right)\\
			\leq& C(1+\| \nabla \hat{\Bu}\|_{L^2 (\hat{\omega}) })\|\hat{\Bu}\|_{W^{1,2}{(\hat{\omega})}} .
	\end{aligned}
\end{equation}
Note that 
\[\hat{\Bu}(s_1,s_2)=\Bu_R(s_1,\zeta(s_1)+s_2) ,~\hat{p}(s_1,s_2)=p_R(s_1,\zeta(s_1)+s_2) \]
and 
\[
 \partial_{s_1}\hat{\Bu}=\partial_{y_1}\Bu_R+ \zeta' \partial_{y_2}\Bu_R,~\partial_{s_2}\hat{\Bu}=\partial_{y_2}\Bu_R ,~\partial_{s_1}\hat{p}=\partial_{y_1}p_R+ \zeta' \partial_{y_2}p_R,~\partial_{s_2}p=\partial_{y_2}p.
\]
Hence it follows from \eqref{7-9} that we have 
\begin{equation}
	\begin{aligned}
		\|\Bu^r\|_{L^\infty( R^{-1}(\omega') )}=&\|\Bu_R\|_{L^\infty( \omega' )}=\|\hat{\Bu}\|_{L^\infty(\hat{\omega}')}\\
		\leq& C(1+\| \nabla \hat{\Bu}\|_{L^2 (\hat{\omega}) })\|\hat{\Bu}\|_{W^{1,2}{(\hat{\omega})}}  \\
		\le& C(1+\| \nabla \Bu_R\|_{L^2 (\omega) })\|\Bu_R\|_{W^{1,2}{(\omega)}} \\
		\leq & C(1+\| \nabla \Bu^r\|_{L^2 (\Omega_{t-\beta^*f(t),t}^r) })\|\Bu^r\|_{W^{1,2}{(\Omega_{t-\beta^*f(t),t}^r)}}  ,
	\end{aligned}
\end{equation}
where $R^{-1}(\omega')=\{z:~Rz \in \omega'\}$. Finally, we have 
\begin{equation*}
\begin{aligned}
	|\Bu(x^*)|=&r^{-1} |\Bu^r(z^*)| \leq r^{-1}\|\Bu^r\|_{L^\infty(R^{-1}(\omega'))}\\
\leq&Cr\|\nabla\Bu\|_{L^2(\Omega_{t-\beta^*f(t),t})}^2+  C\|\nabla\Bu\|_{L^2(\Omega_{t-\beta^*f(t),t})}\\
&+C\|\nabla\Bu\|_{L^2(\Omega_{t-\beta^*f(t),t})}\| \Bu\|_{L^2(\Omega_{t-\beta^*f(t),t})}+ Cr^{-1}\| \Bu\|_{L^2(\Omega_{t-\beta^*f(t),t})}\\
\leq& \frac{C}{f(x_1^*)}.
\end{aligned}
\end{equation*}
This finishes the proof of the proposition.
\end{proof}
Combining Proposition \ref{interior decay} and \ref{boundary decay}, we obtain Theorem \ref{decay rate}.




\medskip

\begin{thebibliography}{99}

	
	\bibitem{A1} C. J. Amick, Steady solutions of the Navier-Stokes equations in unbounded channels and pipes, {\em Ann. Scuola Norm. Sup. Pisa Cl. Sci.}, {\bf 4}(1977), 473--513.
	
	\bibitem{A2} C. J. Amick, Properties of steady Navier-Stokes solutions for certain unbounded channels and pipes, {\em Nonlinear Anal.}, {\bf 2}(1978), 689--720.
	
	\bibitem{AF} C. J. Amick and L. E. Fraenkel, Steady solutions of the Navier-Stokes equations representing plane flow in channels of various types, {\em Acta Math.}, {\bf 144}(1980), 83--151.
	
	
	
	
	
	
	
	\bibitem{Bo} M. E. Bogovski\u{i}, Solution of the first boundary value problem for the equation of continuity of an incompressible medium, {\em Dokl. Akad. Nauk SSSR}, {\bf248}(1979), no. 5, 1037--1040.
	

	


	
	
	
\bibitem{Ga} G. P. Galdi, {\em An introduction to the mathematical theory of the Navier-Stokes equations: Steady-state problems}. Springer, New-York, 2011.
	

\bibitem{GT} D. Gilbarg and N. Trudinger, {\it  Elliptic Partial Differential Equations of Second Order},  2nd edition, Springer-Verlag: Berlin, 1984.
	
\bibitem{HLP} G. H. Hardy, J. E.  Littlewood, and G. P\'{o}lya,  {\it Inequalities}, 2nd edition, Cambridge University Press, 1952.


\bibitem{Horgan} C. O. Horgan, Plane entry flows and energy estimates for the Navier-Stokes equations, {\em Arch. Ration. Mech. Anal.}, {\bf 68}(1978), 359--381.

\bibitem{HW} C. O. Horgan and L. T. Wheeler, Spatial decay estimates for the Navier-Stokes equations with application to the problem of entry flow, {\em SIAM J. Appl. Math.}, {\bf 35}(1978), 97--116.
	
	
\bibitem{KP0} L. V. Kapitanskii and  K. Pileckas, On spaces of solenoidal vector fields and boundary value problems for the Navier-Stokes equations in domains with
noncompact boundaries, {\em Trudy Mat. Inst. Steklov}, {\bf 159}(1983), 5--36.


 \bibitem{KP} K. Kaulakyt\.{e} and K. Pileckas, On the nonhomogeneous boundary value problem for the Navier-Stokes system in a class of unbounded domains, {\em J. Math. Fluid Mech.}, {\bf 14}(2012), no.4, 693--716.
	

 \bibitem{LS0} O. A. Ladyzhenskaja and V. A. Solonnikov, On some problems of vector
analysis and generalized formulations of boundary value problems for the Navier-Stokes equations, {\em Zapiski Nauchn. Sem. LOMI}, {\bf 59}(1976), 81-116.
	
\bibitem{LS} O. A. Ladyzhenskaja and V. A. Solonnikov, Determination of solutions of boundary value problems for stationary Stokes and Navier-Stokes equations having an unbounded Dirichlet integral, {\em Zap. Nauchn. Sem. Leningrad. Otdel. Mat. Inst. Steklov. (LOMI)}, {\bf 96}(1980), 117--160.
	
	

	\bibitem{MF1} H. Morimoto and H. Fujita, On stationary Navier-Stokes flows in 2D semi-infinite channel involving the general outflow condition, in: Navier-Stokes Equations and Related Nonlinear Problems, Ferrara, 1999, {\em Ann. Univ. Ferrara Sez. VII (N. S.)}, {\bf 46}(2000), 285--290.
	
	\bibitem{MF2} H. Morimoto and H. Fujita, A remark on the existence of steady Navier-Stokes flows in a certain two-dimensional infinite channel, {\em Tokyo J. Math.}, {\bf 25}(2002), no.2, 307--321.
	
	
	
	
	\bibitem{Na} C. L. M. H. Navier, M\'{e}moire sur les Lois du Mouvement des Fluides, {\em Mem. Acad. Sci. Inst. de France}, {\bf 6}(1823), 389--440.

	\bibitem{NP} S. A. Nazarov and K. Pileckas, Asymptotics of solutions to Stokes and Navier-Stokes equations in domains with paraboloidal outlets to infinity, {\em Rend. Sem. Mat. Univ. Padova}, {\bf 99}(1998), 1--43.

\bibitem{P0} K. Pileckas, Existence of solutions for the Navier-Stokes equations having
an infinite dissipation of energy, in a class of domains with noncompact boundaries, {\em Zapiski Nauchn. Sem. LOMI}, {\bf 110}(1981), 180-202.

\bibitem{P1} K. Pileckas, Strong solutions of the steady nonlinear Navier-Stokes system in domains with exits to infinity, {\em Rend. Sem. Mat. Univ. Padova}, {\bf 97}(1997), 235--267.

 \bibitem{P2} K. Pileckas, Asymptotics of solutions of the stationary Navier–Stokes system of equations in a domain of layer type, {\em Sb. Math.}, {\bf 193}(2002), no.12, 1801--1836.

\bibitem{P3} K. Pileckas, Asymptotics of solutions to the Navier-Stokes system with nonzero flux in a layer-like domain, {\em Asymptot. Anal.}, {\bf 69}(2010), no.3-4, 219--231.


\bibitem{S1} V. A. Solonnikov, On the solvability of boundary and initial-boundary value problems for the Navier-Stokes equations in domains with noncompact boundaries, {\em Zap. Nauchn. Sem. Leningrad. Otdel. Mat. Inst. Steklov. (LOMI)}, {\bf 96}(1980), 288--293, 312.

\bibitem{S2} V.A. Solonnikov, On the solvability of boundary and initial-boundary value problems for the Navier-Stokes system in domains with noncompact boundaries, {\em Pacific J. Math.}, {\bf93}(1981), no.2, 443--458.

\bibitem{S3} V.A. Solonnikov, Solutions of stationary Navier-Stokes equations with an infinite Dirichlet integral, {\em Zap. Nauchn. Sem. Leningrad. Otdel. Mat. Inst. Steklov. (LOMI)}, {\bf 115}(1982), 257--263, 311.

\bibitem{SP} V. A. Solonnikov and K. Pileckas, Certain spaces of solenoidal vectors and
the solvability of the boundary value problem for the Navier-Stokes system of equations in domains with noncompact boundaries, {\em Zapiski Nauchn.
Sem. LOMI}, {\bf 73}(1977), 136-151.

\bibitem{SWX22} K. Sha, Y. Wang, and C. Xie, On the Leray problem for steady flows in two-dimensional infinitely long channels with slip boundary conditions, preprint, 2022, arXiv:2210.16833. 
	
\bibitem{WX1} Y. Wang and C. Xie, Uniform structural stability of Hagen–Poiseuille flows in a pipe, {\it  Comm. Math. Phys.}, {\bf 393}(2022), no. 3, 1347--1410.
	
\bibitem{WX2} Y. Wang and C. Xie, Existence and asymptotic behavior of large axisymmetric solutions for steady Navier-Stokes system in a pipe, {\em Arch. Ration. Mech. Anal.}, {\bf 243}(2022), no. 3, 1325--1360. 

 
 
	
	
\end{thebibliography}
\end{document}